\documentclass[12pt]{amsart}

\usepackage[utf8]{inputenc}
\usepackage{amsfonts,amsthm,amsmath,amssymb,mathrsfs} 
\usepackage{pgf,tikz}            
\usepackage[colorlinks=true,citecolor=magenta,linkcolor=blue,urlcolor=blue]{hyperref}       
\usepackage[font=footnotesize]{caption}
\usepackage[shortlabels]{enumitem}
\usepackage[margin=1.25in]{geometry}  
\usepackage{setspace}                 

\usepackage[capitalize,nameinlink,noabbrev]{cleveref} 

\newcommand{\ZZ}{{\mathbb Z}}
\newcommand{\ol}[1]{\overline{#1}}
\DeclareMathOperator{\wild}{wld}
\DeclareMathOperator{\clb}{clb}
\DeclareMathOperator{\cub}{cub}

\newcommand{\standout}[1]{\textbf{#1}}    

\tikzset{highlight/.style={opacity=0.25,line width=10},wild/.style={dotted,ultra thick}}

\theoremstyle{plain}
\newtheorem{theorem}{Theorem}[section]
\newtheorem{proposition}[theorem]{Proposition}
\newtheorem{lemma}[theorem]{Lemma} 
\newtheorem{corollary}[theorem]{Corollary}

\theoremstyle{definition}
\newtheorem{example}[theorem]{Example}

\newtheorem{definition}[theorem]{Definition}

\newtheorem{question}[theorem]{Question}

\newtheorem{construction}[theorem]{Construction}  
\newtheorem{alg}[theorem]{Algorithm}

\begin{document}

\title[Wild Number]{The Wild Number of an Edge-Colored Graph}
\author[Anders]{Katie Anders}
\address{Department of Mathematics, University of Texas at Tyler, Tyler, TX, 75799}
\email{kanders@uttyler.edu}
\author[Foster-Greenwood]{Briana Foster-Greenwood}
\address{Department of Mathematics and Statistics, California State Polytechnic University, Pomona, CA, 91768}
\email{brianaf@cpp.edu}
\author[Garcia]{Rebecca Garcia}
\address{Department of Mathematics and Computer Science, Colorado College, Colorado Springs, CO, 80903}
\email{regarcia@coloradocollege.edu}
\author[krawzik]{Naomi Krawzik}
\address{Department of Mathematics and Statistics, Sam Houston State University, Huntsville, TX, 77320}
\email{krawzik@shsu.edu}
\date{\today}

\begin{abstract}
We introduce the wild number of an edge-colored graph as a measure of how close an edge-colored graph is to having a spanning tree in every color. This combinatorial concept originates in the algebraic theory of generalized graph splines. After showing that determining the wild number of a graph is an NP-complete problem, we provide bounds on the wild number and find the exact wild number for trees, cycles, and families of graphs with restrictions on the edge-colorings. This article serves as an invitation to the topic of wild numbers and includes several open problems, many of which are suitable for undergraduate research projects. 
\end{abstract}

\maketitle



\section{The Wild Number}\label{sec:introduction}

Suppose you represent a city with a graph where each vertex represents a location in that city, and each edge represents a transportation route. As an example, consider the graph in \cref{fig:components}. 
The orange routes (shown in subgraph $G_1$) are 
only suitable for biking; the purple routes (shown in subgraph $G_2$) are only suitable for driving; and the green routes (shown in subgraph $G_3$) are
only suitable for walking. The city would like to improve its connectivity by making some edges accessible to all modes of transportation. Which edges should be upgraded? What is the minimum number of edges that would need to be upgraded in order for a person to be able to travel between any two locations using entirely their preferred mode of transportation? To answer these and other questions arising from similar scenarios, we developed the notion of the wild number of an edge-colored graph as described below.

\subsection{Terminology and Notation} Let $G = (V,E)$ be a finite connected graph and let $\mathcal{C}$ be a set of colors. 
Define an edge-coloring on $G$ to be a function $\gamma: E(G) \rightarrow \mathcal{C}$.   An edge-colored graph $(G,\gamma)$ is a connected graph $G$ together with an edge-coloring $\gamma$.  In this paper, we will assume $\gamma$ is surjective, but we do not assume that $\gamma$ is a proper edge coloring. Throughout this paper, we use $(G, \gamma)$ to denote a graph $G$ with vertex set $V(G)$, edge set $E(G)$, and edge-coloring $\gamma:E(G)\rightarrow\mathcal{C}$.  When context is clear, we often refer to the edge-colored graph simply as $G$ rather than as $(G,\gamma)$.  We let $n=|V(G)|, m=|E(G)|$, and $\ell=|\mathcal{C}|$.

For any color $c_i\in\mathcal{C}$, let $E_i(G)=\{e\in E(G):\gamma(e)=c_i\}$, and let $G_i$ be the graph with vertex set $V(G)$ and edge set $E_i(G)$. For any $1\leq i\leq \ell$, we use the notation $\kappa(G_i)$ for the number of connected components of the monochromatic subgraph $G_i$.  We let $\kappa_{\mathcal{C}}(G)=\sum_{i=1}^{\ell}\kappa(G_i)$, the sum of the number of components of all monochromatic subgraphs $G_1, G_2, \ldots, G_{\ell}$.

\begin{figure}
    \centering
    \begin{minipage}{.5\textwidth}\centering
        \begin{tikzpicture}
        \def \radius {8pt};
        \draw (-.75,.75) node {$G:$};
        \draw[blue,highlight,scale=1.5] (0,0)--(1,0)--(1,1);
        \draw[blue,highlight,scale=1.5] (2,1)--(3,1)--(3,0);
        \draw[orange,highlight,scale=1.5] (0,0)--(0,1);
        \draw[orange,highlight,scale=1.5] (1,0)--(2,0)--(2,1);
        \draw[teal,highlight,scale=1.5] (0,1)--(1,1)--(2,1);
        \draw[teal,highlight,scale=1.5] (2,0)--(3,0);
       \draw[gray,fill=white] (1.5*0,1.5*1) circle[radius=\radius] node {\footnotesize $v_{1}$};
       \draw[gray,fill=white] (1.5*1,1.5*1) circle[radius=\radius] node {\footnotesize $v_{2}$};
       \draw[gray,fill=white] (1.5*2,1.5*1) circle[radius=\radius] node {\footnotesize $v_{3}$};
       \draw[gray,fill=white] (1.5*3,1.5*1) circle[radius=\radius] node {\footnotesize $v_{4}$};
       \draw[gray,fill=white] (1.5*3,1.5*0) circle[radius=\radius] node {\footnotesize $v_{5}$};
       \draw[gray,fill=white] (1.5*2,1.5*0) circle[radius=\radius] node {\footnotesize $v_{6}$};
       \draw[gray,fill=white] (1.5*1,1.5*0) circle[radius=\radius] node {\footnotesize $v_{7}$};
       \draw[gray,fill=white] (1.5*0,1.5*0) circle[radius=\radius] node {\footnotesize $v_{8}$};
    \end{tikzpicture}
    \end{minipage}\begin{minipage}{.5\textwidth}\centering
        \begin{tikzpicture}
        \draw (-.75,.5) node {$G_1:$};
        \draw[orange,highlight,scale=1] (0,0)--(0,1);
        \draw[orange,highlight,scale=1] (1,0)--(2,0)--(2,1);
        \foreach \x in {0,1,2,3}
          \foreach \y in {0,1}
             \draw[gray,fill=white] (\x,\y) circle[radius=5pt] {};
    \end{tikzpicture}\vskip36pt\begin{tikzpicture}
    \draw (-.75,.5) node {$G_2:$};
        \draw[blue,highlight,scale=1] (0,0)--(1,0)--(1,1);
        \draw[blue,highlight,scale=1] (2,1)--(3,1)--(3,0);
        \foreach \x in {0,1,2,3}
          \foreach \y in {0,1}
             \draw[gray,fill=white] (\x,\y) circle[radius=5pt] {};
    \end{tikzpicture}\vskip36pt
    \begin{tikzpicture}
    \draw (-.75,.5) node {$G_3:$};
        \draw[teal,highlight,scale=1] (0,1)--(1,1)--(2,1);
        \draw[teal,highlight,scale=1] (2,0)--(3,0);
        \foreach \x in {0,1,2,3}
          \foreach \y in {0,1}
             \draw[gray,fill=white] (\x,\y) circle[radius=5pt] {};
    \end{tikzpicture}
    \end{minipage}
    
    \caption{A graph $G$ with $3$ edge colors and its monochromatic subgraphs $G_1$, $G_2$, and $G_3$. There are five orange components, four purple components, and five green components.}
    \label{fig:components}
\end{figure}
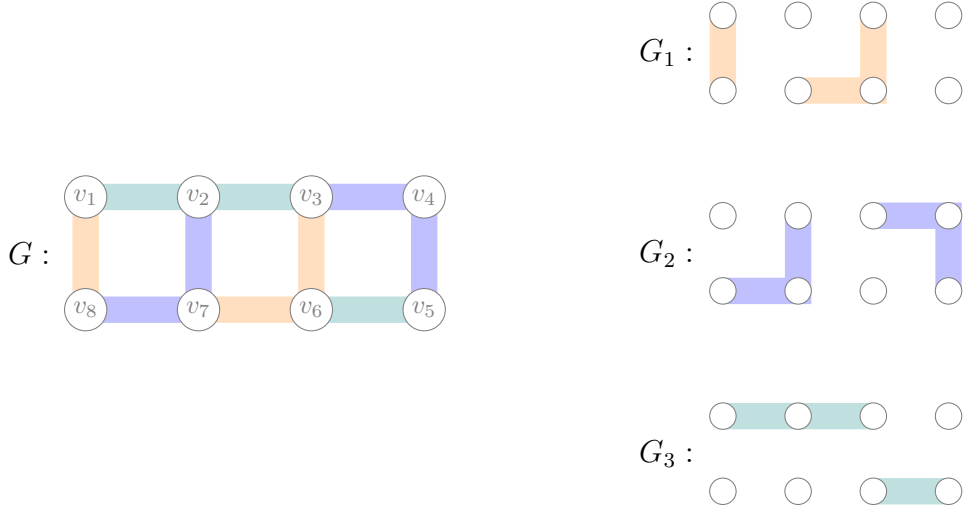

Let $W\subseteq E(G)$.  For any color $c_i$, define the graph $G_i^W$ to be the subgraph of $G$ with vertex set $V(G)$ and edge set $W\cup E_i(G)$.  If the set of edges $W$ is such that for every $c_i\in C$, the subgraph $G_i^W$ is connected, then $W$ is a \textbf{color-connecting set}.  
We say that we \textbf{color an edge $e$ wild} if $e \in W$. A color-connecting set of minimal  size is an \textbf{ideal wild set}. The \textbf{wild number} of $G$, denoted $\wild(G)$, is the size of an ideal wild set of $G$. In other words, the wild number is the minimum number of edges that need to be colored wild for the graph to have a spanning tree in every color. 

Coloring an edge wild in $G$ can connect components for some colors, thus bringing the graph closer to having a spanning tree in those colors. The color $c_i$ needs $\kappa(G_i)-1$ connections since at least that many wild edges are needed for $G_i^W$ to be connected.
 We say that the edge $uv$ \standout{helps} the color $c_i$ if coloring $uv$ wild connects two components in the subgraph $G_i$, that is, $\kappa(G_i^{ uv }) < \kappa(G_i)$.  
The \standout{dip number} of an edge $e$ is \[\delta(e)=\#\{i:\text{edge $e$ helps color $c_i$}\}=\sum_{i=1}^{\ell}(\kappa(G_i)-\kappa(G_i^{e})),\]
which tells us how many colors we can help by coloring that edge wild.
The dip number of a set of edges $W$ is $\sum_{i=1}^{\ell}(\kappa(G_i)-\kappa(G_i^W))$. Note that $\delta(W)\leq\sum_{e\in W}\delta(e)$.

  \begin{example}
  
    To illustrate these definitions, consider the graph $G$ in \cref{fig:components}, along with its monochromatic subgraphs $G_1$ with edge set $E_1=\{v_1v_8,v_3v_6,v_6v_7,\}$; $G_2$ with edge set $E_2=\{v_2v_7,v_7v_8,v_4v_5\}$; and $G_3$ with edge set $E_3=\{v_1v_2,v_2v_3,v_5v_6\}$. 
    We assign the colors $c_1$, $c_2$, and $c_3$ to be the colors orange, purple, and green, respectively. 
    Then we have $\kappa(G_1)=5$, $\kappa(G_2)=4$, and $\kappa(G_3)=5$. 
    Orange and green each need $4$ connections, while purple needs $3$ connections.
   Note that $\kappa_{\mathcal{C}}(G)=\kappa(G_1)+\kappa(G_2)+\kappa(G_3)=5+4+5=14$, the total number of connected components across all monochromatic subgraphs of $G$, and the total number of connections needed for this graph is $11$. If edge $v_1v_2$ were colored wild, the number of components in $G_1$ and $G_2$ would decrease by one, while the number of components in $G_3$ would remain the same. Therefore, the dip number of edge $v_1v_2$ is $\delta(v_1v_2)=2$. 
   Let $W=E(G)-\{v_1v_2, v_2v_3, v_3v_4,v_3v_6\}$.  One can see that this is a color-connecting wild set.  The dip lower bound, described in \cref{sec:bounds}, lets us know that this is the smallest size possible of a color-connecting wild set for this graph, and therefore $W$ is an ideal wild set and $\wild(G) = 6$.  
\end{example}


\subsection{Organization} The goal of this paper is to introduce the wild number of an edge-colored graph and provide some foundational results.   In \cref{sec:classic}, we provide a couple of potential applications and connections to other established mathematical concepts with the goal of encouraging our dear readers to consider their own applications and connections to the wild number of $G$. In \cref{sec:families}, we investigate the wild number of families of graphs such as trees and cycles. We show that the decidability question of whether a graph has a wild number bounded above by $k$ is an NP-hard problem in \cref{sec:complexity} and explore various upper and lower bounds for the wild number in \cref{sec:bounds}. In  \cref{sec:colors}, we provide results for extremal cases on the number of colors used in coloring the edges. We end the paper with a set of open questions in \cref{sec:conclusion}.


\section{Connections to Other Mathematical  Concepts}\label{sec:classic}

As context and motivation for our work, we first relate wild numbers of edge-colored graphs to problems in algebraic and classical graph theory. 

\subsection{Connection to Generalized Algebraic Graph Splines}
Our novel concept of the wild number of an edge-colored graph was motivated initially by previous work by Anders et al.~\cite{REUFGroup} on edge-labeled graphs that admit only constant splines. Classically, a spline is a piecewise polynomial function on a polyhedral complex where the polynomials agree up to some degree of smoothness at the intersections of the faces. This concept has since been generalized 
to be defined over edge-labeled graphs and is studied in numerical analysis, algebraic geometry, and homological algebra \cite{Gilbert, Schenck}. Given a ring $R$ and a graph $G(V,E)$, an \textbf{edge-labeling} on $G$ is a function $\alpha: E(G)\rightarrow I(R)$, where $I(R)$ is the 
set of ideals in $R$. A \textbf{generalized graph spline} on the edge-labeled graph $(G,\alpha)$ is a function $\rho:V(G)\rightarrow R$ such that for any edge $uv\in E(G)$, the difference $\rho(u)-\rho(v)$ is in the ideal $\alpha(uv)$.

A \textbf{constant spline} on an edge-labeled graph $(G,\alpha)$ over a ring $R$ is a spline $\rho$ such that for some $r\in R$, we have $\rho(v)=r$ for all $v\in V(G)$.
Any edge-labeled graph admits constant splines, since choosing any fixed element of the ring as the assignment for each vertex in the graph produces a generalized graph spline. Moreover, note that if an edge is labeled as the $0$-ideal, then the two vertices incident to that edge must be identically labeled in any generalized graph spline. In~\cite{REUFGroup}, Anders et al.\ determine conditions on the edge-labeling of graphs over $\mathbb{Z}_m$ that force all of the generalized graph splines to be constant for a given edge-labeled graph. 
The result \cite[Corollary 3.10]{REUFGroup} states that for each factor $p_i^{e_i}$ in the prime factorization of $m$, the graph must contain a spanning tree $T_i$ such that all edge labels of $T_i$ are contained in the ideal $(p_i^{e_i})$.  When the graph contains such a spanning tree for each factor in the prime factorization of $m$, the graph admits only constant splines.

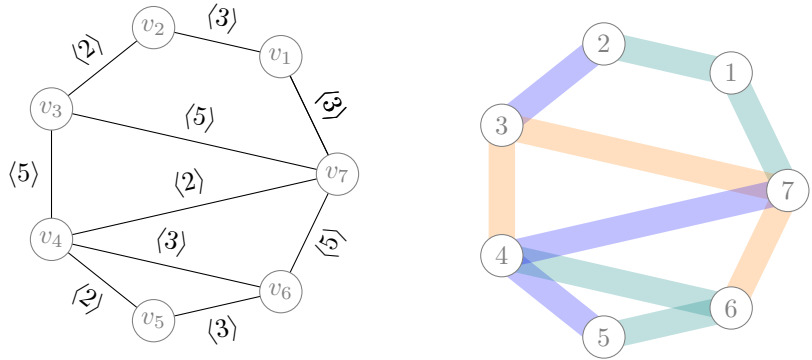
\begin{figure}
      \begin{tikzpicture}
        \def \radius {8pt};
        \def \va {(1*360/7:2)};
        \def \vb {(2*360/7:2)};
        \def \vc {(3*360/7:2)};
        \def \vd {(4*360/7:2)};
        \def \ve {(5*360/7:2)};
        \def \vf {(6*360/7:2)};
        \def \vg {(7*360/7:2)};
        \draw \va--\vb node[midway,sloped,above] {\footnotesize $\langle 3 \rangle $};
        \draw \vb--\vc node[midway,sloped,above] {\footnotesize $\langle 2 \rangle$};
        \draw \vc--\vd node[midway,left] {\footnotesize $\langle 5 \rangle$};
        \draw \vd--\ve node[midway,sloped,below] {\footnotesize $\langle 2 \rangle$};
        \draw \ve--\vf node[midway,sloped,below] {\footnotesize $\langle 3 \rangle $};
        \draw \vf--\vg node[midway,sloped,below] {\footnotesize $\langle 5 \rangle$};
        \draw \vg--\va node[midway,sloped,above] {\footnotesize $\langle 3 \rangle $};
        \draw \vg--\va node[midway,sloped,above] {\footnotesize $\langle 3 \rangle $};
        \draw \vc--\vg node[midway,sloped,above] {\footnotesize $\langle 5 \rangle$};
        \draw \vd--\vf node[midway,sloped,above] {\footnotesize $\langle 3 \rangle $}; 
        \draw \vd--\vg node[midway,sloped,above] {\footnotesize $\langle 2 \rangle$};

        \foreach \pt/\name in {\va/{v_1},\vb/{v_2},\vc/{v_3},\vd/{v_4}, \ve/{v_5}, \vf/{v_6}, \vg/{v_7}}
         \draw[gray,fill=white] \pt circle[radius=\radius] node {\footnotesize ${\name}$};
    \end{tikzpicture} \hspace{.5in} \begin{tikzpicture}
        \def \radius {8pt};
        \def \va {(1*360/7:2)};
        \def \vb {(2*360/7:2)};
        \def \vc {(3*360/7:2)};
        \def \vd {(4*360/7:2)};
        \def \ve {(5*360/7:2)};
        \def \vf {(6*360/7:2)};
        \def \vg {(7*360/7:2)};
        \draw[teal,highlight] \va--\vb; 
        \draw[blue,highlight] \vb--\vc; 
        \draw[orange,highlight] \vc--\vd; 
        \draw[blue,highlight] \vd--\ve; 
        \draw[teal,highlight] \ve--\vf; 
        \draw[orange,highlight] \vf--\vg; 
        \draw[teal,highlight] \vg--\va; 
        \draw[orange,highlight] \vc--\vg; 
        \draw[teal,highlight] \vd--\vf; 
        \draw[blue,highlight] \vd--\vg; 

        \foreach \pt/\name in {\va/{1},\vb/{2},\vc/{3},\vd/{4}, \ve/{5}, \vf/{6}, \vg/{7}}
         \draw[gray,fill=white] \pt circle[radius=\radius] node {\footnotesize ${\name}$};
    \end{tikzpicture} 
    \caption{A graph with edges labeled by ideals in $\ZZ_{30}$ and its corresponding edge-colored graph. On the left, the edge label $\langle 2 \rangle$ corresponds with coloring the edge \texttt{purple}, $\langle 3 \rangle$ with \texttt{green} and $\langle 5 \rangle$ with \texttt{orange}.} 
    \label{constant spline}
\end{figure}

A natural question to explore is how far away a graph is from having only constant splines. Specifically, while working on \cite{REUFGroup}, Alissa Crans posed the question: ``Given an edge-labeled graph $(G, \alpha)$, what is the minimum number of edges one must relabel as the $0$-ideal to force the graph to admit only constant splines?'' 
In the next example, we illustrate how, for certain graph labelings, this is equivalent to the question ``What is the minimum number of edges that need to be colored wild in order for the graph to have a spanning tree in every color?''

\begin{example}
    Consider the edge-labeled graph  $(G, \alpha)$ illustrated in \cref{constant spline}, where $R = \mathbb{Z}_{30}$. An example of a non-constant spline is $(11,8,12,2,8,17,2)$, where the $i$-th coordinate is $\rho(v_i)$. Consider the set of edges $W=\{v_2v_3, v_3v_4, v_5v_6, v_1v_7\}$.  In the graph on the left, if we label these edges with the $0$-ideal, then this graph will admit only constant splines. Similarly, if we color these edges wild in the graph on the right, there will be a spanning tree in every color. It turns out that this could not be done with less edges (as 4 is the ceiling lower bound, which will be discussed in \cref{sec:bounds}) and so the answer to both questions posed above is $4$.  
\end{example}
\subsection{Connection to Edge-Disjoint Spanning Trees and the Spanning Tree Packing Number of a Graph}\label{sub: EDST}
With the concept of the wild number of a graph established, there are a host of research directions one can develop simply by varying the given conditions that characterize this notion. In particular, among the defining parameters are (i) the finite graph $G$, (ii) the edge-coloring $\gamma$, (iii) the wild number $\wild(G)$, and (iv) the number of colors $\ell$. Thus, as an example, one direction of research, described in \cref{sec:bounds} and \cref{sec:colors} of this paper, is to determine (bounds for) the wild number of a given graph $G$ equipped with an edge-coloring $\gamma$. 

In this subsection, we consider the following question: Given a graph $G$ and a desired wild number, say $\wild(G) = w$, determine the existence and possible construction of an edge-coloring $\gamma: E(G) \to \mathcal{C}$ that produces the desired wild number. In particular, we briefly explore this question here for $\wild(G) = 0$. As it turns out, this investigation connects the concept of the wild number of a graph to two classical topics in graph theory: (a) edge-disjoint spanning trees and (b) Hamiltonian decomposition of a graph.

Edge-disjoint spanning trees, introduced by Nash-Williams \cite{Nash-W} and Tutte \cite{Tutte}, provide an equivalent characterization for when a given graph $G$ would have an edge-coloring that yields $\wild(G) = 0$. We restate the theorem below:

\begin{theorem}[Theorem 1 in \cite{Nash-W}]\label{EDST}
    A graph $G(V,E)$ has $\ell$ edge-disjoint spanning trees if and only if \[\vert E_{\mathbf{P}}(G) \vert \geq \ell \left( \vert \mathbf{P} \vert - 1\right), \] for every partition $\mathbf{P}$ of $V$.
\end{theorem}

Here, $\mathbf{P}$ is a partition of the vertex set $V$ and the set $E_{\mathbf{P}}(G)$ is the set of edges in $G$ that join vertices belonging to different parts of the partition $\mathbf{P}$. Because of \cref{EDST}, we know precisely when a graph $G$ has $\ell$ edge-disjoint spanning trees, and thus we know precisely when there exists an edge-coloring $\gamma: E(G) \to \{c_1, c_2,\ldots, c_\ell \}$ that produces a wild number 0 for $G$. In particular, when $G$ satisfies the inequality in \cref{EDST}, then there exist $\ell$ edge-disjoint spanning trees $T_{1},T_2, \ldots, T_{\ell}$. Thus, we can define an edge-coloring
\[ 
\gamma(e)= \begin{cases} 
      c_i & \text{if $e \in E(T_{i})$} \\
      c_1 & \text{if $e \notin E(T_{i})$ for all $i$,}
      \end{cases}
\]
which yields the following observation:

\begin{theorem}
    A graph $G(V,E)$ has an edge-coloring $\gamma: E(G) \to \{c_1,c_2, \ldots, c_\ell \}$ such that $\wild(G) = 0$ if and only if \[\vert E_{\mathbf{P}}(G) \vert \geq \ell \left( \vert \mathbf{P} \vert - 1\right), \] for every partition $\mathbf{P}$ of $V$.
\end{theorem}

As noted in network design and communication protocols, having a high number of edge-disjoint spanning trees, i.e., the spanning tree packing number of a graph $G$ ${\rm STP}(G)$, corresponds to the idea of having an edge-coloring on $G$ with a high number of colors such that the wild number is zero. This implies that the wild number of an edge-colored graph can be considered as another indicator of how close a graph is to being more resilient to edge failures or attacks, where the ideal wild sets point to the edges in a network that require extra security so that the network maintains its connectivity, even with multiple edge removals.

We can also produce colorings that yield wild number zero using Hamiltonian decompositions.
Recall that a graph $G$ has a {Hamiltonian decomposition} if there is a partition of the edge set $E(G)$ into Hamiltonian cycles. Similar to the previous discussion on edge-disjoint spanning trees, a graph $G$ with a Hamiltonian decomposition given by the partition of the edge set $E(G) = E_{1} \cup \cdots \cup E_{\ell}$ induces a coloring $\gamma: E(G) \to \{ c_{1}, \ldots, c_{\ell} \}$ of the edges of $G$ in such a way that the resulting graph has wild number 0. In this manner, one can also think of these wild concepts as a generalization of a Hamiltonian decomposition. 

\section{Graph Families and Operations}\label{sec:families}
In this section we discuss some families of graphs for which the wild number is easily computed.  It is worth noting here that if $\gamma$ is not surjective then the wild number is $n-1$. Additionally, if $\ell=1$ then the wild number is zero. Thus, throughout this section, we assume $\ell=|\mathcal{C}|\geq 2$.
\subsection{Trees}
 We begin with some results about the wild number for edge-colored trees. Recall that a bridge in a connected graph is an edge whose deletion disconnects the graph. The next lemma shows that every bridge in a graph must be colored wild to ensure there is a spanning tree in every color. 
\begin{lemma}\label{bridges}
   Let $G$ be an edge-colored graph with $\ell\geq2$.  For any bridge $e$ in $E(G)$, $e$ must belong to $W$ for every color-connecting wild set $W$ of $G$. 
\end{lemma}
\begin{proof}
Suppose $W$ is a color-connecting wild set of $G$ and there exists a bridge $e$ in $E(G)$ such that $e$ is of color $c_j$ and $e\not\in W$.  Then for any color $c_i$ with $i\neq j$, $G_i^W$ is not connected, so $W$ cannot be a color-connecting wild set for $G$.  
\end{proof}

\begin{corollary} 
    If $T$ is an order $n$ tree with $\ell\geq2$ edge colors, then $\wild(T)=n-1$.
\end{corollary}
\begin{proof} 
    Every edge in $T$ is a bridge. Apply \cref{bridges}.
\end{proof}

\subsection{Cycles}

We next determine the wild number for cycle graphs. The result depends only on the number of vertices and colors, not on the particular edge-coloring.
An example is shown in \cref{fig:cycles}.

\begin{proposition}\label{prop:cycle wild number}
    If $C_n$ is an edge-colored cycle with $n$ vertices and $\ell$ colors, then
    \[\wild(C_n)=\begin{cases}
        n-2 & \text{if $\ell=2$} \\
        n-1 & \text{if $\ell>2$.}
    \end{cases}\]
\end{proposition}
\begin{proof}

    Suppose $\ell=2$ and $\mathcal{C}=\{c_1,c_2\}$. Let $e_1$ be an edge in $C_n$ with color $c_1$ and $e_2$ be an edge in $C_n$ with color $c_2$. Let $W=E-\{e_1,e_2\}$. Then $|W\cup \{e_i\}|=n-1$ for $i=1,2$, 
    so both $G_1^W$ and $G_2^W$ are connected, and $W$ is a color-connecting set. 
    Thus $\wild(C_n)\leq n-2$ when $\ell=2$.
    
    To show that $\wild(G)$ is not less than $n-2$, suppose $W$ is a color-connecting set of size $n-3$. Then for $G_1^W$ to be connected, two of the three remaining edges not in $W$ must be color $c_1$, but for $G_2^W$ to be connected, we would also need two more edges of color $c_2$. This gives a total of four more edges, but there are only three edges not in $W$, so this is impossible.
    
   Now suppose $\ell\geq 3$ and let $W$ be a set of $n-2$ edges, i.e., a set containing all but two edges of $G$. Since $\ell\geq 3$, there exists a color $c_i$ such that $E_i\subseteq W$. Then $|E_i\cup W|=|W|<n-1$, so $G_i^W$ is not connected. Thus there is no color-connecting wild set of size $n-2$, and we can conclude $\wild(G)=n-1$.
\end{proof}

\begin{figure}
    \centering
    {}\hfill
        \begin{tikzpicture}
        \def \radius {6pt};
        \def \va {(1*60:2)};
        \def \vb {(2*60:2)};
        \def \vc {(3*60:2)};
        \def \vd {(4*60:2)};
        \def \ve {(5*60:2)};
        \def \vf {(6*60:2)};
        \draw[blue,highlight] \va--\vb;
        \draw[blue,highlight] \vb--\vc;
        \draw[orange,highlight] \vc--\vd;
        \draw[blue,highlight] \vd--\ve;
        \draw[orange,highlight] \ve--\vf;
        \draw[orange,highlight] \vf--\va;
        \draw[wild] \vf--\va--\vb;
        \draw[wild] \vc--\vd--\ve;
        \foreach \pt/\name in {\va/{1},\vb/{2},\vc/{3},\vd/{4},\ve/{5},\vf/{6}}
         \draw[gray,fill=white] \pt circle[radius=\radius];
    \end{tikzpicture}\hfill
    \begin{tikzpicture}
        \def \radius {6pt};
        \def \va {(1*60:2)};
        \def \vb {(2*60:2)};
        \def \vc {(3*60:2)};
        \def \vd {(4*60:2)};
        \def \ve {(5*60:2)};
        \def \vf {(6*60:2)};
        \draw[blue,highlight] \va--\vb;
        \draw[blue,highlight] \vb--\vc;
        \draw[orange,highlight] \vc--\vd;
        \draw[teal,highlight] \vd--\ve;
        \draw[orange,highlight] \ve--\vf;
        \draw[orange,highlight] \vf--\va;
        \draw[wild] \va--\vb--\vc--\vd--\ve--\vf;
        \foreach \pt/\name in {\va/{1},\vb/{2},\vc/{3},\vd/{4},\ve/{5},\vf/{6}}
        \draw[gray,fill=white] \pt circle[radius=\radius];
    \end{tikzpicture}\hfill{}
    \caption{Ideal wild sets for cycle graphs. Dotted edges are wild. Two colors (left): choose all but one edge of each color to be wild. More than two colors (right): choose all but one edge to be wild.}
    \label{fig:cycles}
\end{figure}
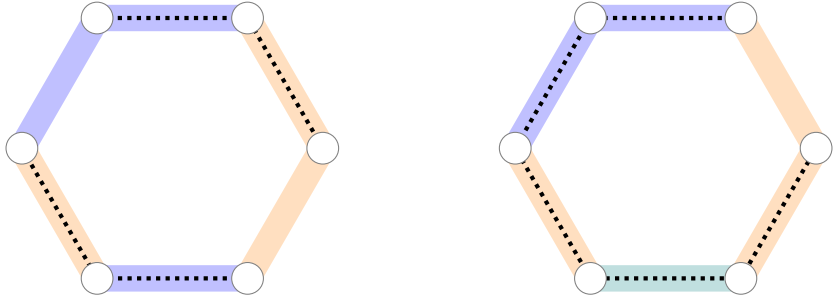

\subsection{Vertex Amalgamation}

Let $G$ and $H$ be connected graphs. If $u$ and $v$ are vertices in $G$ and $H$, respectively, then the \standout{vertex amalgamation} $(G\cup H)/\{u=v\}$ is the graph obtained from $G\cup H$ by identifying vertex $u$ with vertex $v$. Intuitively, we think of gluing $G$ and $H$ together at a vertex.

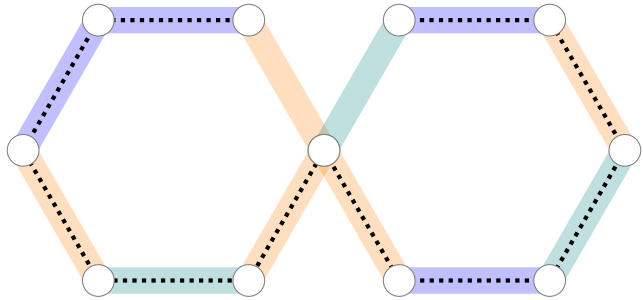
\begin{figure}
    \centering
    \begin{tikzpicture}
        \def \radius {6pt};
        \def \va {(1*60:2)};
        \def \vb {(2*60:2)};
        \def \vc {(3*60:2)};
        \def \vd {(4*60:2)};
        \def \ve {(5*60:2)};
        \def \vf {(6*60:2)};
        \draw[blue,highlight] \va--\vb;
        \draw[blue,highlight] \vb--\vc;
        \draw[orange,highlight] \vc--\vd;
        \draw[teal,highlight] \vd--\ve;
        \draw[orange,highlight] \ve--\vf;
        \draw[orange,highlight] \vf--\va;
        \draw[wild] \va--\vb--\vc--\vd--\ve--\vf;
        \foreach \pt/\name in {\va/{1},\vb/{2},\vc/{3},\vd/{4},\ve/{5},\vf/{6}}
         \draw[gray,fill=white] \pt circle[radius=\radius];
        \draw[shift=({4,0}),blue,highlight] \va--\vb;
        \draw[shift=({4,0}),teal,highlight] \vb--\vc;
        \draw[shift=({4,0}),orange,highlight] \vc--\vd;
        \draw[shift=({4,0}),blue,highlight] \vd--\ve;
        \draw[shift=({4,0}),teal,highlight] \ve--\vf;
        \draw[shift=({4,0}),orange,highlight] \vf--\va;
        \draw[shift=({4,0}),wild] \vf--\va--\vb;
        \draw[shift=({4,0}),wild] \vc--\vd--\ve--\vf;
         \foreach \pt/\name in {\va/{1},\vb/{2},\vc/{3},\vd/{4},\ve/{5},\vf/{6}}
         \draw[shift=({4,0}),gray,fill=white] \pt circle[radius=\radius];
    \end{tikzpicture}
    \caption{Amalgamation of two cycles of wild number $5$. The amalgamation has wild number $10$.}
    \label{fig:amalgamation}
\end{figure}

\begin{proposition}\label{prop:amalgamation}
   Let $(G,\gamma_G)$ and $(H,\gamma_H)$ be edge-colored graphs. If $u$ and $v$ are vertices in $G$ and $H$, respectively, then the wild number of the vertex amalgamation graph $(G\cup H)/\{u=v\}$ is the sum of the wild numbers of $G$ and $H$, i.e.,
    \[\wild((G\cup H)/\{u=v\})=\wild(G)+\wild(H),\]
    where the wild numbers of $G$ and $H$ are computed with respect to the set of colors of $G\cup H$.
\end{proposition}

\begin{proof}
    Let $A=(G\cup H)/\{u=v\}$. Since $G$ and $H$ share no edges, any spanning tree of $A$ is the amalgamation of a spanning tree for $G$ and a spanning tree for $H$, and any color-connecting wild set $W\subseteq E(A)$ can be expressed as a disjoint union $W=W_G\cup W_H$ for some color-connecting wild sets  $W_G\subseteq E(G)$ and $W_H\subseteq E(H)$. 
 Since $W_G$ and $W_H$ are disjoint, $|W|$ is minimized when $|W_G|$ and $|W_H|$ are minimized.
 Thus $\wild(A)=\wild(G)+\wild(H)$.
\end{proof}

Recall that a \standout{block} in a graph is a maximal connected subgraph, and any graph $G$ can be expressed as an amalgamation of its blocks, glued at the cut vertices of $G$. Repeated application of \cref{prop:amalgamation} allows us to determine the wild number of a graph based on the wild numbers of the blocks that compose the graph. An example is shown in \cref{fig:amalgamation}.

\begin{corollary}
    If $G$ is a graph with blocks $B_1,B_2,\ldots,B_k$, then the wild number of the graph $G$ is the sum of the wild numbers of the blocks, i.e.,
    \[\wild(G)=\wild(B_1)+\wild(B_2)+\cdots+\wild(B_k),\]
    where the wild number of each block is computed with respect to the set of colors for $G$.
\end{corollary}

\begin{proof}
    Repeatedly apply \cref{prop:amalgamation}.
\end{proof}

\section{Complexity}\label{sec:complexity}

In this section, we show the problem of determining whether a graph has an ideal wild set with at most  $k$ edges is NP-complete.  We refer to this decision problem as $k$-WILD:
\begin{quote}
  \standout{$k$-WILD}. Given an edge-colored multigraph $(G,\gamma)$ and a nonnegative integer $k$, does there exist an ideal wild set $W\subseteq E(G)$ with $|W|\leq k$?
\end{quote}
To prove $k$-WILD is NP-hard, we show that $3$-SAT (a known NP-complete problem) reduces to $k$-WILD.

We first recall the setup for the $3$-SAT problem.
Let $X=\{x_1,x_2,\ldots,x_k\}$ be a set of Boolean variables and let $\ol{X}=\{\ol{x}_1,\ol{x}_2,\ldots,\ol{x}_k\}$  be the set of their negations. 
The elements of $U=X\cup \ol{X}$ are called \standout{literals}. 
A \standout{clause over $X$} is a disjunction of three distinct literals (e.g., $x_1\lor\ol{x}_2\lor x_4$). Let $S(\ell,k)$ be the set of statements that can be expressed as a conjunction of $\ell$ clauses over a set of $k$ variables, where without loss of generality we assume
\begin{itemize}
    \item the clauses are distinct (do not have the same set of literals);
    \item a variable and its negation do not appear in the same clause; and
    \item each variable appears in at least one clause.
\end{itemize}
The $3$-SAT decision problem is:
\begin{quote}
    \standout{$3$-SAT.} Given a statement $F$ in $S(\ell,k)$, does there exist an assignment of truth values to the variables that makes the statement $F$ true (i.e., is $F$ \standout{satisfiable})?
\end{quote}
To show $3$-SAT reduces to $k$-WILD, we start with a statement $F$ and construct (in polynomial time) a graph $\mathcal{G}(F)$ that has wild number $k$ if and only if $F$ is satisfiable.

\begin{construction}\label{gadgetgraph}
Let $F\in S(\ell,k)$ be a statement with $\ell$ clauses  $C_1,C_2,\ldots,C_{\ell}$ over the set of variables $X=\{x_1,x_2,\ldots,x_k\}$, and let $U=X\cup \ol{X}$. Write $C_i=L_{i1}\lor L_{i2}\lor L_{i3}$, where $\{L_{i1},L_{i2},L_{i3}\}\subseteq U$. Define an edge-colored multigraph $\mathcal{G}(F)$ with $\ell+2$ colors $c_1,c_2,\ldots,c_{\ell+2}$ with the following vertices and edges:

\begin{description}
\item[Vertices] Add
   vertices $C_i$ for $1\leq i\leq \ell$; 
   vertices $x_j$, $\overline{x_j}$, and $y_j$ for $1\leq j\leq k$; and
   vertex $z$. 

\item[Edges] For all $1\leq i,i'\leq \ell$, $1\leq j\leq k$ and $u_j\in \{x_j,\ol{x}_j\}$, add
\begin{enumerate}[(a)]
    \item\label{type a} edge $C_iu_j$ in color $c_i$ if $u_j$ appears in clause $C_i$;
    \item\label{type b} edges $y_ju_j$ and $zu_j$ in color $c_i$ if $u_j$ does not appear in clause $C_i$; 
    \item\label{type c} edge $C_{i'}u_j$ in color $c_i$ if $u_j$ appears in clause $C_{i'}$ but not in clause $C_i$;
    \item\label{type d} edges $C_iL_{i1}$, $C_{i}L_{i2}$, $C_{i}L_{i3}$, and $zu_j$ in color $c_{\ell+1}$; and
    \item\label{type e} edges $C_iL_{i1}$, $C_{i}L_{i2}$, $C_{i}L_{i3}$, $zu_j$, and $y_ju_j$ in color $c_{\ell+2}$.
\end{enumerate}
\end{description}
We refer to the edges formed in steps \ref{type a}--\ref{type e} as edges of types \ref{type a}--\ref{type e}.
An edge may be added in multiple colors, yielding a multigraph with parallel edges. 

\end{construction}

Before giving an example of the construction, we describe the connected components in each color.

\begin{lemma}\label{crazygraphcomponents}
    Let $G=\mathcal{G}(F)$, where $F$ is a statement in $S(\ell,k)$. 
    For $1\leq i\leq \ell$, the subgraph $G_i$ has two components: $\{C_i,L_{i1},L_{i2},L_{i3}\}$ and $V(G)-\{C_i,L_{i1},L_{i2},L_{i3}\}$.
    For $i=\ell+1$, the subgraph $G_{\ell+1}$ has $k+1$ components: $\{y_1\}$, $\{y_2\}$, \ldots, $\{y_k\}$, and $V(G)-\{y_1,y_2,\ldots,y_k\}$. For $i=\ell+2$, the subgraph $G_{\ell+2}$ is connected.
\end{lemma}

\begin{proof}
    For color $c_{\ell+2}$, each vertex is in the same component as $z$, so $G_{\ell+2}$ is connected. For color $c_{\ell+1}$, each vertex $y_j$ ($1\leq j\leq k$) is isolated since there are no edges of color $c_{\ell+1}$ incident to $y_j$ in $G$. The remaining vertices of $G$ are in the same component as $z$ due to the edges of type \ref{type d}.

    Now consider color $c_i$ with $1\leq i\leq \ell$.
    We will show the subgraph $G_i$ has exactly two components: $G_{i,C_i}$ containing vertex $C_i$ and $G_{i,z}$ containing vertex $z$.
    Let $U_i=\{L_{i1},L_{i2},L_{i3}\}$, the set of literals appearing in clause $C_i$.   We know $U_i\subseteq G_{i,C_i}$ due to the edges of type \ref{type a}, and $U-U_i\subseteq G_{i,z}$ due to the edges of type \ref{type b}. Since a literal and its negation do not appear in the same clause, each vertex $y_j$ ($1\leq j\leq k$) has an edge of color $c_i$ to $x_j$ or $\ol{x}_j$ (possibly both). Hence $\{y_1,y_2,\ldots,y_k\}\subseteq G_{i,z}$. 
    Now consider vertex $C_{i'}$ with $i'\neq i$. Since the clauses $C_i$ and $C_{i'}$ are distinct, there exists $u\in U_{i'}-U_i\subseteq U-U_i$. Then by \ref{type c} and \ref{type b} we have edges $C_{i'}u$ and $zu$ in color $c_i$, so $C_{i'}\in G_{i,z}$.
    Finally, the vertices $C_i$ and $z$ are not in the same $c_i$-component because if $u\in U_i$, then $u$ is incident to only one edge of color $c_i$, namely the edge $C_iu$. 
\end{proof}

\begin{example}\label{ex:construction}
    Let $F$ be the conjunction of the three clauses 
       \[
        C_1 = x_1\lor \ol{x}_2\lor x_4,\quad
        C_2 = \ol x_1\lor x_3 \lor \ol x_4, \text{ and } 
        C_3 = \ol x_2 \lor x_3 \lor x_4
    \]
    over a set of four Boolean variables $X=\{x_1,x_2,x_3,x_4\}$.
Here, $\ell=3$ and $k=4$, so following \cref{gadgetgraph}, we add to our graph $\mathcal{G}(F)$ the vertices \[C_1, C_2, C_3, x_1, x_2, x_3, x_4, \ol x_1, \ol x_2, \ol x_3, \ol x_4, y_1, y_2, y_3, y_4, z.\]
We next add colored edges to the graph.  
We have edge colors $c_1,c_2,c_3$ corresponding to the three clauses and two additional edge colors $c_4$ and $c_5$.  

We begin with the edges of color $c_1$, illustrated in \cref{fig:crazygraph1}. Our edges of type \ref{type a} are $C_1x_1$, $C_1\ol x_2$, $C_1x_4$.  Our edges of type \ref{type b} are $y_1\ol x_1$, $y_2x_2$, $y_3x_3$, $y_3\ol x_3$, $y_4\ol x_4$ and  $z\ol x_1$, $zx_2$, $zx_3$, $z\ol x_3$, $zx_4$.  Lastly, we add in edges of type \ref{type c}. We have $U_2-U_1=\{\ol x_1,x_3,\ol x_4\}$ and $U_3-U_1=\{x_3\}$, so we add edges $C_2\ol{x}_1$, $C_2x_3$, $C_2\ol{x}_4$ and edge $C_3x_3$ in color $c_1$. 

Similarly, we add all the edges of color $c_2$.  Our edges of type \ref{type a} are $C_2\ol x_1$, $C_2x_3$, $C_2\ol x_4$.  Our edges of type \ref{type b} are $y_1x_1$, $y_2x_2$, $y_2\ol x_2$, $y_3\ol x_3$, $y_4x_4$ and $zx_1$, $zx_2$, $z\ol x_2$, $z\ol x_3$, $zx_4$.  Lastly, we add in edges of type \ref{type c}.  We have $U_1-U_2=\{x_1, \ol x_2, x_4\}$ and $U_3-U_2=\{\ol x_2, x_4\}$, so we add edges $C_1x_1$, $C_1\ol x_2$, $C_1x_4$ and edges $C_3\ol x_2$, $C_3x_4$ in color $c_2$.

We now add all the edges of color $c_3$. Our edges of type \ref{type a} are $C_3\ol x_2$, $C_3x_3$, $C_3x_4$.  Our edges of type \ref{type b} are $y_1 x_1$, $y_1 \ol x_1$, $y_2 x_2$, $y_3 \ol x_3$, $y_4 \ol x_4$ and  $z x_1$, $z \ol x_1$, $z x_2$, $z \ol x_3$, $z\ol x_4$.  Next, we add in edges of type \ref{type c}. We have $U_1-U_3=\{x_1\}$ and $U_2-U_3=\{\ol x_1, \ol x_4\}$, so we add the edge $C_1 x_1$ and edges $C_2 \ol x_1$, $C_2\ol{x}_4$ in color $c_3$.

Next, we add edges in the extra color $c_{4}$. 
The edges of type \ref{type d} are
\[C_1x_1, C_1\ol{x}_2, C_1x_4, 
C_2\ol{x}_1,C_2x_3,C_2\ol{x}_4,
C_3\ol{x}_2, C_3x_3, C_3x_4,\]
and 
\[zx_1,z\ol{x}_1,zx_2,z\ol{x}_2,zx_3,z\ol{x_3},zx_4,z\ol{x}_4.\]
Finally, the edges of type \ref{type e} in the extra color $c_5$ are the same as those of type \ref{type d} along with
\[
x_1y_1, y_1\ol{x}_1, x_2y_2, y_2\ol{x}_2, x_3y_3, y_3\ol{x}_3, x_4y_4, y_4\ol{x}_4.
\]
The underlying simple graph in \cref{fig:crazygraph} summarizes all the edges in $\mathcal{G}(F)$, without indication of colors and parallel edges.

\end{example}
\begin{figure}
    \centering
    \begin{tikzpicture}
        \def \radius {8pt};
        \def \xone {(1,0)};
        \def \yone {(2,0)};
        \def \nxone {(3,0)};
        \def \xtwo {(5,0)};
        \def \ytwo {(6,0)};
        \def \nxtwo {(7,0)};
        \def \xthr {(9,0)};
        \def \ythr {(10,0)};
        \def \nxthr {(11,0)};
        \def \xfour {(13,0)};
        \def \yfour {(14,0)};
        \def \nxfour {(15,0)};
        \def \Cone {(4,2.75)};
        \def \Ctwo {(8,2.75)};
        \def \Cthr {(12,2.75)};
        \def \root {(8,-2.75)};
        \draw[blue,thick] \yone--\nxone;
        \draw[blue,thick] \xtwo--\ytwo;
        \draw[blue,thick] \xthr--\ythr;
        \draw[blue,thick] \nxthr--\ythr;
        \draw[blue,thick] \nxfour--\yfour;
        \foreach \pt in {\xone,\nxtwo,\xfour}
        \draw[blue,ultra thick] \Cone -- \pt;
        \foreach \pt in {\nxone,\xthr,\nxfour}
        \draw[blue,thick] \Ctwo -- \pt;
        \foreach \pt in {\xthr}
        \draw[blue,thick] \Cthr -- \pt;
        \foreach \x in {3,5}
        \draw[bend right,blue,thick] (\x,0) to (8,-2.75);
        \foreach \x in {1,5,7}
        \draw[bend left,blue,thick] (16-\x,0) to (8,-2.75);
        \foreach \pt/\name in {\xone/{x_1},\yone/{y_1},\nxone/{\ol{x}_1},\xtwo/{x_2},\ytwo/{y_2},\nxtwo/{\ol{x}_2},\xthr/{x_3},\ythr/{y_3},\nxthr/{\ol{x}_3},\xfour/{x_4},\yfour/{y_4},\nxfour/{\ol{x}_4},\root/{z},\Cone/{C_1},\Ctwo/{C_2},\Cthr/{C_3}}
        \draw[fill=white] \pt circle[radius=\radius] node {\footnotesize ${\name}$};
    \end{tikzpicture}
    \caption{The two components in color $c_1$ for the graph $\mathcal{G}(F)$ in \cref{ex:construction}. The first connected component of $G_{1}$ comprises the thicker edges, while the thinner edges compose the second connected component.}
    \label{fig:crazygraph1}
\end{figure}
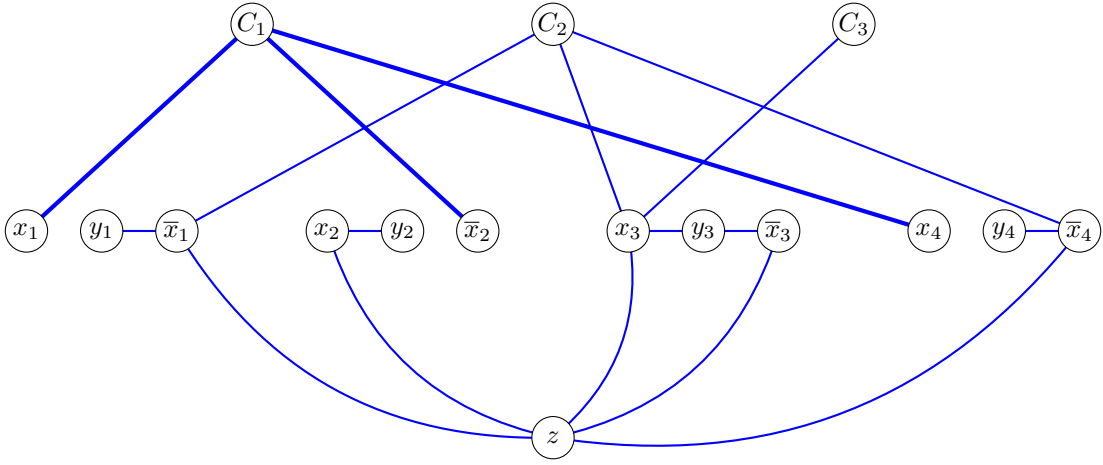

\begin{figure}
    \centering
    \begin{tikzpicture}
        \def \radius {8pt};
        \def \xone {(1,0)};
        \def \yone {(2,0)};
        \def \nxone {(3,0)};
        \def \xtwo {(5,0)};
        \def \ytwo {(6,0)};
        \def \nxtwo {(7,0)};
        \def \xthr {(9,0)};
        \def \ythr {(10,0)};
        \def \nxthr {(11,0)};
        \def \xfour {(13,0)};
        \def \yfour {(14,0)};
        \def \nxfour {(15,0)};
        \def \Cone {(4,2.75)};
        \def \Ctwo {(8,2.75)};
        \def \Cthr {(12,2.75)};
        \def \root {(8,-2.75)};

        \draw \xone -- \yone -- \nxone;
        \draw \xtwo -- \ytwo -- \nxtwo;
        \draw \xthr -- \ythr -- \nxthr;
        \draw \xfour -- \yfour --\nxfour;
        \foreach \pt in {\xone,\nxtwo,\xfour}
        \draw \Cone -- \pt;
        \foreach \pt in {\nxone,\xthr,\nxfour}
        \draw \Ctwo -- \pt;
        \foreach \pt in {\nxtwo,\xthr,\xfour}
        \draw \Cthr -- \pt;
        \foreach \x in {1,3,5,7}
        {\draw[bend right] (\x,0) to (8,-2.75);
        \draw[bend left] (16-\x,0) to (8,-2.75);}
        \foreach \pt/\name in {\xone/{x_1},\yone/{y_1},\nxone/{\ol{x}_1},\xtwo/{x_2},\ytwo/{y_2},\nxtwo/{\ol{x}_2},\xthr/{x_3},\ythr/{y_3},\nxthr/{\ol{x}_3},\xfour/{x_4},\yfour/{y_4},\nxfour/{\ol{x}_4},\root/{z},\Cone/{C_1},\Ctwo/{C_2},\Cthr/{C_3}}
        \draw[fill=white] \pt circle[radius=\radius] node {\footnotesize ${\name}$};
    \end{tikzpicture}
    \caption{Underlying simple graph corresponding to statement $F=C_1\land C_2\land C_3$ with clauses $C_1=x_1\lor \ol{x}_2\lor x_4$, $C_2=\ol{x}_1\lor x_3\lor\ol{x}_4$, and $C_3=\ol{x}_2\lor x_3\lor x_4$.}
    \label{fig:crazygraph}
\end{figure}
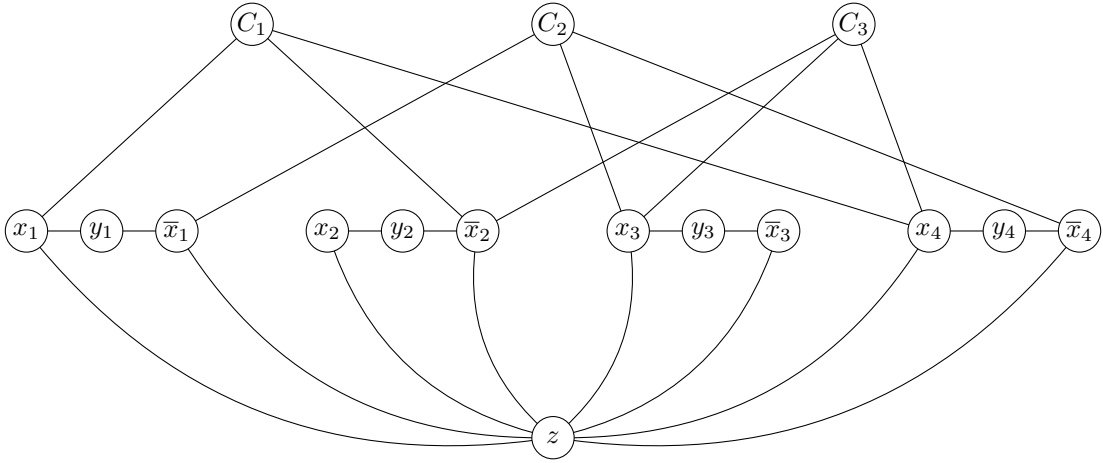

The next lemma tells us what edges are necessary to connect the components of $\mathcal{G}(F)$ in the extra color $c_{\ell+1}$.

\begin{lemma}\label{pair}
    If $F$ is a statement in $S(\ell,k)$ and $W$ is a color-connecting wild set for the graph $\mathcal{G}(F)$, then for all $1\leq j\leq k$, the set $W$ includes at least one edge from the pair $\{y_jx_j,y_j\overline{x_j}\}$.
\end{lemma}

\begin{proof}
    Let $G=\mathcal{G}(F)$. Let $1\leq j\leq k$ and let $W$ be a color-connecting wild set for $G$. Since $y_j$ is an isolated vertex in $G_{\ell+1}$, the only way for $G_{\ell+1}^W$ to be connected is if $W$ includes an edge incident to $y_j$ in $G$, i.e., $W$ includes an edge from the pair $\{y_jx_j, y_j\ol{x}_j\}$. 
\end{proof}

Now, we determine the edges to connect the components of $\mathcal{G}(F)$ in the first $\ell$ colors.

\begin{lemma}\label{GiW-connected}
    Let $G=\mathcal{G}(F)$, where $F$ is a statement in $S(\ell,k)$.
    Given an assignment of truth values to the variables $x_1,x_2,\ldots,x_k$, let $u_j$ be the element of $\{x_j,\ol{x}_j\}$ that is true and let $W=\{y_ju_j:1\leq j\leq k\}\cap E_{\ell+2}$. For $1\leq i\leq \ell$, the graph $G_i^W$ is connected if and only if clause $C_i$ is true.
\end{lemma}

\begin{proof}
    Let $1\leq i\leq \ell$. Recall from \cref{crazygraphcomponents} that the two components of $G_i$ are the component $G_{i,C_i}$ and $G_{i,z}$ and that $y_j\in G_{i,z}$ for all $1\leq j\leq k$. Thus $G_i^W$ is connected if and only if there is some true literal $u_j$ in the component $G_{i,C_i}$. The literals in component $G_{i,C_i}$ are precisely the literals appearing in clause $C_i$. It follows that $G_i^W$ is connected if and only if clause $C_i$ is true. 
\end{proof}

Finally, we can relate the wild number of the graph $\mathcal{G}(F)$ to satisfiability of the statement $F$.

\begin{theorem}\label{SATiffWILD}
    For every statement $F\in S(\ell,k)$, there exists an edge-colored multigraph $G$ with $\ell+2$ colors such that statement $F$ is satisfiable if and only if the multigraph $G$ has wild number $k$.
\end{theorem}

\begin{proof}
    Let $F\in S(\ell,k)$ be a statement with $\ell$ clauses  $C_1,C_2,\ldots,C_{\ell}$ over a set of variables $X=\{x_1,x_2,\ldots,x_k\}$. Let $G=\mathcal{G}(F)$ as in \cref{gadgetgraph}.
    
    Suppose $F$ is satisfiable 
    and consider an assignment of truth values to the variables $x_1,x_2,\ldots,x_k$ so that each clause of $F$ is true.
    For $1\leq j\leq k$, let $u_j$ be the element of 
    $\{x_j,\ol{x}_j\}$ that is True. 
    Let $W=\{y_ju_j:1\leq j\leq k\}\cap E_{\ell+2}$. For $1\leq i\leq \ell$, we know clause $C_i$ is true, so by \cref{GiW-connected}, the graph $G_i^W$ is connected. Furthermore, for $i=\ell+1$, we have connected the $c_{\ell+1}$-components $\{y_j\}$ to the $c_{\ell+1}$-component $V-\{y_1,y_2,\ldots,y_k\}$ via the wild edges $y_ju_j$, so $G_{\ell+1}^W$ is connected. And, of course, $G_{\ell+2}^W=G_{\ell+2}$ is connected. 
    
    Next, we need to show $W$ is ideal.
    Note that the set $W$ consists of exactly $k$ edges (since $y_ju_j\in W$ if and only if $y_j\ol{u}_j\notin W$), so $\wild(G)\leq k$. But also, the graph $G$ has $k+1$ components in color $c_{\ell+1}$, so $\wild(G)\geq k$. 
    Hence $\wild(G)=k$.

    Conversely, suppose $G$ has wild number $k$. Let $W$ be an ideal wild set for $G$. Without loss of generality, we can assume $W\subseteq E_{\ell+2}$ (since otherwise we may substitute an edge of $W$ with a parallel edge of color $c_{\ell+2}$). By \cref{pair}, the set $W$ must contain at least one edge from each pair $\{y_jx_j,y_j\overline{x}_j\}$ ($1\leq j\leq k$). In fact, since $|W|=k$, we know $W$ contains \textit{exactly} one edge from each pair. 
    For $1\leq j\leq k$, let $u_j$ be the element in $\{x_j,\ol{x}_j\}$ such that $y_ju_j\in W$, and set $u_j$ to True (and $\ol{u}_j$ to False). 
    Since $W$ is a color-connecting set, we know $G_i^W$ is connected for all $1\leq i\leq \ell$,
    so each clause $C_i$ is True by \cref{GiW-connected}.
    Thus, the statement $F$ is satisfiable.
\end{proof}

\begin{theorem}
    The $k$-WILD problem is NP-complete.
\end{theorem}
\begin{proof}
    The $k$-WILD problem belongs to NP since given a certificate $W$ listing $k$ wild edges, a verifier can check in polynomial time whether $\kappa(G_i^W)=1$ for each color $c_i$. Additionally, given a statement $F\in S(\ell,k)$, the graph $\mathcal{G}(F)$ can be constructed in polynomial time, so by the reduction from $3$-SAT to $k$-WILD in \cref{SATiffWILD}, we conclude $k$-WILD is NP-hard. Hence $k$-WILD is NP-complete.
\end{proof}

\section{Bounds}\label{sec:bounds}
We showed in \cref{sec:complexity} that the decision variant of determining the wild number of a graph is NP-hard. Nevertheless, we have various upper and lower bounds that can be used to estimate the wild number of a given edge-colored graph. Additionally, we provide a greedy algorithm that constructs a (not necessarily optimal) color-connecting wild set.

\subsection{Bounds} We begin with the component bounds, which are convenient as they are quickly computed  using the colored components of the graph. 


\begin{proposition}[Component Bounds]\label{component-bounds}
    If $(G,\gamma)$ is an edge-colored graph with $\ell$ colors, then
    \[\max_{1\leq i\leq \ell}(\kappa(G_i)-1)\leq \wild(G)\leq\sum_{i=1}^{\ell}(\kappa(G_i)-1).\]
\end{proposition}
\begin{proof}
    If $W$ is an ideal wild set for $G$, then for each color $c_i$, there exists a set $W_i\subseteq W$ such that $W_i$ is a minimum $c_i$-connecting set. Since $|W_i|\leq|W|$ for each $i$, we have $\max_{1\leq i\leq \ell}(\kappa(G_i)-1)\leq \wild(G)$.

    To prove the upper bound, consider that for each color $c_i$, there exists a set $W_i$  with $|W_i|=\kappa(G_i)-1$ so that $G_i^{W_i}$ is connected. Then $W=\bigcup_{i=1}^{\ell} W_i$ is a color-connecting set for $G$, so $\wild(G)\leq|W|\leq\sum_{i=1}^{\ell}(\kappa(G_i)-1)$. 
\end{proof}

The bounds for $\wild(G)$ in the statement of \cref{component-bounds} will be used throughout this section and are given names in the following definition.
\begin{definition}
    The \standout{component lower bound of $(G,\gamma)$}, denoted $\operatorname{clb}(G)$, is given by
    \[ \operatorname{clb}(G) = \max_{1\leq i\leq \ell}(\kappa(G_i)-1). \]
    Similarly, the \standout{component upper bound of $(G,\gamma)$}, denoted $\operatorname{cub}(G)$, is given by
    \[\operatorname{cub}(G) =  \sum_{i=1}^{\ell}(\kappa(G_i)-1).\]
\end{definition}

\begin{example}
    As an example of computing these bounds, we return to the graph in \cref{fig:components}.  Recall that for this graph, $\kappa(G_1)=5, \kappa(G_2)=4$, and $\kappa(G_3)=5$.  Thus $\operatorname{clb}(G)=5-1=4$, while $\operatorname{cub}(G)=4+3+4=11$.
\end{example}

In addition to these component bounds, we have also developed two other lower bounds for $\wild(G)$, known as the ceiling lower bound of $G$ and the dip lower bound of $G$. The ceiling lower bound is quicker to compute than the dip lower bound, but the dip lower bound is sometimes an improvement since it is always greater than or equal to the ceiling lower bound. 
We begin with the ceiling lower bound. 

\begin{proposition}[Ceiling Lower Bound]\label{ceiling bound}
    If $(G,\gamma)$ is an edge-colored graph with $\ell$ colors, then 
    \[
   \wild(G)\geq\left \lceil\frac{\cub(G)}{\ell-1} \right\rceil=\left\lceil \frac{\sum_{i=1}^{\ell}(\kappa(G_i)-1)}{\ell-1}\right\rceil.
    \]
\end{proposition}
\begin{proof}
Let $W$ be a color-connecting wild set of $G$. Then \[|W|\cdot(\ell-1)\geq \delta(W) \geq \kappa_{\mathcal{C}}(G)-\ell,\] 
where $\delta(W)$ is the dip number, as defined in \cref{sec:introduction}.  It follows that $|W|\geq\frac{\kappa_{\mathcal{C}}(G)-\ell}{\ell-1}$ and $\wild(G) \geq \frac{\kappa_{\mathcal{C}}(G)-\ell}{\ell-1}$.
\end{proof}

\begin{definition}
    Given an edge-colored graph $G$, the \standout{ceiling lower bound for $G$} is
    \[
   \left \lceil\frac{\cub(G)}{\ell-1} \right\rceil=\left\lceil \frac{\sum_{i=1}^{\ell}(\kappa(G_i)-1)}{\ell-1}\right\rceil.
    \]
\end{definition}

We now move on to the dip lower bound. Recall that the dip number of an edge is the number of colors helped by coloring the edge wild.  Specifically,
$\delta(uv)=\#\{i:\text{edge $uv$ helps color $c_i$}\}$. The \standout{dip sequence} of a graph is the sequence of dip numbers, arranged in non-increasing order, of all the edges in the graph. Dip sequences will be used here in defining the dip lower bound as well as in \cref{greedy}, which produces a color-connecting wild set for a given edge-colored graph.

\begin{proposition}[Dip Lower Bound]\label{thm:DipLowerBound}
    If $(\delta_1,\delta_2,\ldots,\delta_m)$ is the dip sequence of an edge-colored graph $G$, with $\delta_1\geq\delta_2\geq\cdots\geq\delta_m$, then 
    \begin{equation*}
    \wild(G)\geq\min\left\{p:\delta_1+\delta_2+\cdots+\delta_p\geq\sum_{i=1}^{\ell}(\kappa(G_i)-1)\right\}.
    \end{equation*}
\end{proposition}
\begin{proof}

Let $p_0=\min\left\{p:\delta_1+\delta_2+\cdots+\delta_p\geq\sum_{i=1}^{\ell}(\kappa(G_i)-1)\right\}$. If $W$ is a color-connecting wild set with $p$ edges $e_{i_1},e_{i_2},\ldots,e_{i_p}$, then
\[\delta_1+\delta_2+\cdots+\delta_p\geq 
\delta_{i_1}+\delta_{i_2}+\cdots+\delta_{i_p}\geq\delta(W)\geq\sum_{i=1}^{\ell}(\kappa(G_i)-1).\]
So $p\geq p_0$ and $\wild(G)\geq p_0$.
\end{proof}

\begin{definition}[Dip Lower Bound] 
Given an edge-colored graph $G$, the \standout{dip lower bound for $G$} is 
\[\operatorname{dlb}(G)=\min\left\{p:\delta_1+\delta_2+\cdots+\delta_p\geq\sum_{i=1}^{\ell}(\kappa(G_i)-1)\right\}.\]
\end{definition}

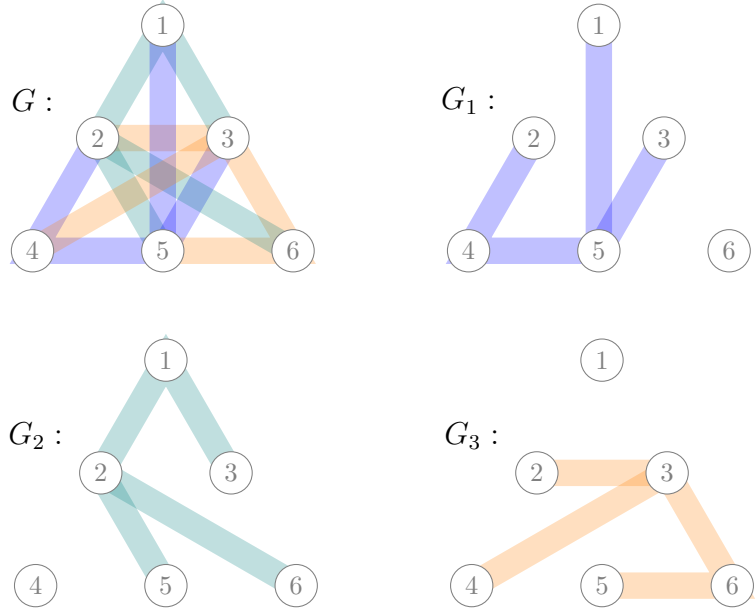
\begin{figure}
    \centering{}
    \begin{tikzpicture}
        \def \radius {8pt};
        \def \va {(30+1*360/3:1)};
        \def \vb {(30+2*360/3:1)};
        \def \vc {(30+3*360/3:1)};
        \def \vd {(90+1*360/3:2)};
        \def \ve {(90+2*360/3:2)};
        \def \vf {(90+3*360/3:2)};
        \def \ga {(90+1*360/6:2)};

        \draw \ga node {$G:$};
        \draw[teal,highlight] \va--\vb;
        \draw[blue,highlight] \vb--\vc;
        \draw[orange,highlight] \vc--\va;
        \draw[blue,highlight] \va--\vd --\vb;
        \draw[orange,highlight] \vb--\ve--\vc;
        \draw[teal,highlight] \vc--\vf--\va;
        \draw[teal,highlight] \va -- \ve;
        \draw[blue,highlight] \vb -- \vf;
        \draw[orange,highlight] \vc --\vd;
        \foreach \pt/\name in {\va/{2},\vb/{5},\vc/{3},\vd/{4},\ve/{6},\vf/{1}}
         \draw[gray,fill=white] \pt circle[radius=\radius] node {\footnotesize ${\name}$};
    \end{tikzpicture} \hspace{0.5in} \begin{tikzpicture}
        \def \radius {8pt};
        \def \va {(30+1*360/3:1)};
        \def \vb {(30+2*360/3:1)};
        \def \vc {(30+3*360/3:1)};
        \def \vd {(90+1*360/3:2)};
        \def \ve {(90+2*360/3:2)};
        \def \vf {(90+3*360/3:2)};
        \def \gb {(90+1*360/6:2)};

        \draw \gb node {$G_1:$};
        \draw[blue,highlight] \vb--\vc;
        \draw[blue,highlight] \va--\vd --\vb;
        \draw[blue,highlight] \vb -- \vf;
        \foreach \pt/\name in {\va/{2},\vb/{5},\vc/{3},\vd/{4},\ve/{6},\vf/{1}}
         \draw[gray,fill=white] \pt circle[radius=\radius] node {\footnotesize ${\name}$};
    \end{tikzpicture} \vskip24pt
    \begin{tikzpicture}
        \def \radius {8pt};
        \def \va {(30+1*360/3:1)};
        \def \vb {(30+2*360/3:1)};
        \def \vc {(30+3*360/3:1)};
        \def \vd {(90+1*360/3:2)};
        \def \ve {(90+2*360/3:2)};
        \def \vf {(90+3*360/3:2)};
        \def \gc {(90+1*360/6:2)};

        \draw \gc node {$G_2:$};

        \draw[teal,highlight] \va--\vb;
       \draw[teal,highlight] \vc--\vf--\va;
       \draw[teal,highlight] \va -- \ve;
        \foreach \pt/\name in {\va/{2},\vb/{5},\vc/{3},\vd/{4},\ve/{6},\vf/{1}}
         \draw[gray,fill=white] \pt circle[radius=\radius] node {\footnotesize ${\name}$};
    \end{tikzpicture} \hspace{0.5in} \begin{tikzpicture}
        \def \radius {8pt};
        \def \va {(30+1*360/3:1)};
        \def \vb {(30+2*360/3:1)};
        \def \vc {(30+3*360/3:1)};
        \def \vd {(90+1*360/3:2)};
        \def \ve {(90+2*360/3:2)};
        \def \vf {(90+3*360/3:2)};
        \def \gc {(90+1*360/6:2)};

        \draw \gc node {$G_3:$};

        \draw[orange,highlight] \vc--\va;
        \draw[orange,highlight] \vb--\ve--\vc;
        \draw[orange,highlight] \vc --\vd;
        \foreach \pt/\name in {\va/{2},\vb/{5},\vc/{3},\vd/{4},\ve/{6},\vf/{1}}
         \draw[gray,fill=white] \pt circle[radius=\radius] node {\footnotesize ${\name}$};
    \end{tikzpicture}
    \caption{A graph $G$ with component lower bound $1$, ceiling lower bound $2$, and dip lower bound $3$. The wild number is $3$, with ideal wild set $W = \{ v_{1}v_{5}, v_{2}v_{6}, v_{3}v_{4}\}$.    }
    \label{fig:bounds}
\end{figure}

The next example illustrates that the component lower bound, ceiling lower bound, and dip lower bound for a graph can all be different.

\begin{example}\label{ex:small pyramid example}
Consider a graph $G$ with vertex set $V=\{v_1,v_2,v_3,v_4,v_5,v_6\}$. Let $c_1$ be the color purple with edge set $E_1=\{v_1v_5,v_2v_4, v_4v_5, v_5v_3\}$, $c_2$ be the color green with edge set $E_2=\{v_2v_6, v_3v_1, v_1v_2, v_2v_5\}$, and $c_3$ be the color orange with edge set $E_3=\{v_3v_4, v_2v_3, v_3v_6, v_6v_5\}$, as shown in \cref{fig:bounds}.  Then $\kappa(G_1)= \kappa(G_2)= \kappa(G_3)=2$, so $\operatorname{clb}(G)=2-1=1, \operatorname{cub}(G)=1+1+1=3$, and $\kappa_{\mathcal{C}}(G)=6$.
The ceiling lower bound is $\left\lceil\frac{\cub(G)}{\ell-1}\right\rceil=\lceil\frac{3}{2}\rceil=2$. To compute the dip lower bound, we use the non-increasing dip sequence 
$1^{(9)}, 0^{(3)}$.  We see that the minimum $p$ of  \cref{thm:DipLowerBound} is $p=3$, so our dip lower bound for this graph is $3$. In this case, the dip lower bound matches the component upper bound, so we may conclude that $\wild(G)=3$.

\end{example}

\subsection{Algorithm for approximating the wild number}

We next describe an algorithm for finding a
color-connecting set for a given edge-colored graph by selecting wild edges one at a time.
In this iterative process, we contract wild edges as they are selected.

To that end, we introduce the following notation. Given an edge-colored graph $(G,\gamma)$ and subset of wild edges $W$, let $G/W$ be the \standout{quotient graph} obtained by contracting the edges in $W$. We denote the vertex set of $G/W$ as $\{[u]:u\in V(G)\}$, where $[u]=[v]$ if and only if there is a wild path from $u$ to $v$ in $G$. Note that $W$ is a color-connecting set for $G$ if and only if $G/W$ has a spanning tree in every color.

When changing edges to wild, intuition tells us that we would like to prioritize changing edges that help as many colors as possible. When there are multiple edges with the highest dip number, we use what we call the potential  of an edge to choose which of these edges to make wild.   
The \standout{potential} of an edge $e$ in an edge-colored graph $G$ is the nonincreasing dip sequence of the graph $G/{e}$.

\begin{example}[Calculating potential]
    Consider the induced subgraph of \cref{fig:counterex} with vertex set $V=\{v_0,v_1,v_2,v_3\}$, orange edge set $E_1 = \{v_0v_1,v_0v_3\}$, purple edge set $E_2 = \{v_0v_2\}$, green edge set $E_3=\{v_1v_2\}$, and red edge set $E_4=\{v_2v_3\}$.
    We call this graph $H$. To calculate the potential of the edge $v_1v_2$, we consider the subgraph $H/v_1v_2$, in which   edge $v_0v_3$ has dip number $3$ (helps green, purple, red), edge $v_0v_1$ has dip number $2$ (helps green, red), edge $v_0v_2$ has dip number $2$ (helps red, green), edge $v_2v_3$ has dip number $2$ (helps purple, green),  and edge $v_1v_2$ now has dip number $0$ (has already been changed to wild). Thus, the edge $v_1v_2$ has potential $(3,2,2,2,0)$. 

\end{example}

In the greedy algorithm, we choose edges to color wild by considering which edges have the greatest dip number and potential.

\begin{alg}[Greedy]\label{greedy} To construct a color-connecting wild set $W$ for an edge-colored graph $(G,\gamma)$:
Initialize by letting $W=\varnothing$.
While the graph $G/W$ is not color-connected, do the following two steps:
\begin{enumerate}
    \item Order the edges in $G/W$ with the highest dip number lexicographically by potential in $G/W$ (i.e., compare first terms of potential, highest wins; if tie, compare next term; etc.).
    \item Select an edge $[u][v]\in E(G/W)$ with greatest potential, and add the edge $uv\in E(G)$ to the wild set $W$.
\end{enumerate}
When $G/W$ is color-connected, return the set $W$.

\end{alg}

\begin{figure}
    \centering{}
    \begin{tikzpicture} 
        \def \radius {8pt};
        \def \va {(90+1*360/3:1)}; 
        \def \vb {(90+2*360/3:1)};
        \def \vc {(90+3*360/3:1)};
        \def \vd {(90+1*360/3:3)};
        \def \ve {(90+2*360/3:3)};
        \def \vf {(90+3*360/3:3)};
        \draw[blue,highlight] \va--\vb;  
        \draw[orange,highlight] \vb--\vc;  
        \draw[blue,highlight] \vc--\va;  
        \draw[blue,highlight] \vd--\ve; 
        \draw[orange,highlight] \ve--\vf;  
        \draw[teal,highlight] \vf--\vd;  
        \draw[teal,highlight] \va -- \vd;  
        \draw[orange,highlight] \vb -- \ve;  
        \draw[teal,highlight] \vc --\vf;  
        \draw[wild] \vb--\va;
        \foreach \pt/\name in {\va/{6},\vb/{5},\vc/{4},\vd/{3},\ve/{2},\vf/{1}}
         \draw[gray,fill=white] \pt circle[radius=\radius] node {\footnotesize ${\name}$};
    \end{tikzpicture}\hfill
    \begin{tikzpicture} 
        \def \radius {8pt};
        \def \va {(90+1*360/3:1)}; 
        \def \vb {(90+1*360/3:1)};
        \def \vc {(90+3*360/3:1)};
        \def \vd {(90+1*360/3:3)};
        \def \ve {(90+2*360/3:3)};
        \def \vf {(90+3*360/3:3)};
        \draw[blue,highlight] \va -- \vb;  
        \draw[orange,highlight] \vb to [bend right=45] (90+3*360/3:1);  
        \draw[blue,highlight] \va to [bend left] (90+3*360/3:1);  
        \draw[blue,highlight] \vd--\ve; 
        \draw[orange,highlight] \ve--\vf;  
        \draw[teal,highlight] \vf--\vd;  
        \draw[teal,highlight] \va -- \vd;  
        \draw[orange,highlight] \vb -- \ve;  
        \draw[teal,highlight] \vc --\vf;  
        \draw[wild] \vf--\ve;
        \foreach \pt/\name in {\va/{6},\vb/{5,6},\vc/{4},\vd/{3},\ve/{2},\vf/{1}}
         \draw[gray,fill=white] \pt circle[radius=1.7*\radius] node {\footnotesize ${\name}$};
    \end{tikzpicture}\hfill\vskip24pt
        \begin{tikzpicture} 
        \def \radius {8pt};
        \def \va {(90+1*360/3:1)}; 
        \def \vb {(90+1*360/3:1)};
        \def \vc {(90+3*360/3:1)};
        \def \vd {(90+1*360/3:3)};
        \def \ve {(90+2*360/3:3)};
        \def \vf {(90+2*360/3:3)};
        \draw[blue,highlight] \va -- \vb;  
        \draw[orange,highlight] \vb to [bend right=45] (90+3*360/3:1);  
        \draw[blue,highlight] \va to [bend left] (90+3*360/3:1);  
        \draw[blue,highlight] \vd --\ve; 
        \draw[orange,highlight] \ve--\vf;  
        \draw[teal,highlight] \vd to [bend right=25] (90+2*360/3:3);  
        \draw[teal,highlight] \va -- \vd;  
        \draw[orange,highlight] \vb -- \ve;  
        \draw[teal,highlight] \vc --\vf;  
        \draw[wild] \va--\vd;
        \foreach \pt/\name in {\va/{6},\vb/{5,6},\vc/{4},\vd/{3},\ve/{2},\vf/{1,2}}
         \draw[gray,fill=white] \pt circle[radius=1.7*\radius] node {\footnotesize ${\name}$};
    \end{tikzpicture}\hfill     
    \begin{tikzpicture} 
        \def \radius {8pt};
        \def \va {(90+1*360/3:1)}; 
        \def \vb {(90+1*360/3:3)};
        \def \vc {(90+3*360/3:1)};
        \def \vd {(90+1*360/3:1)};
        \def \ve {(90+2*360/3:3)};
        \def \vf {(90+2*360/3:3)};
        \draw[orange,highlight] \vb -- (90+3*360/3:1);  
        \draw[blue,highlight] (90+1*360/3:3) to [bend left] (90+3*360/3:1);  
        \draw[blue,highlight] (90+1*360/3:3) --\ve; 
        \draw[orange,highlight] \ve--\vf;  
        \draw[teal,highlight] (90+1*360/3:3) to [bend right] (90+2*360/3:3);  
        \draw[teal,highlight] \va -- \vd;  
        \draw[orange,highlight]  \ve to [bend right] (90+1*360/3:3) ;  
        \draw[teal,highlight] \vc --\vf;  
        \foreach \pt/\name in {\vb/{3,5,6},\vc/{4},\vf/{1,2}}
         \draw[gray,fill=white] \pt circle[radius=1.7*\radius] node {\footnotesize ${\name}$};
    \end{tikzpicture}
    \caption{Top left: Prism graph $G$ with dotted wild edge $v_5v_6$. Top right: Graph $G/\{v_5v_6\}$ with dotted wild edge $[v_1][v_2]$. Bottom left: Graph $G/\{v_5v_6,v_1v_2\}$ with dotted wild edge $[v_3][v_6]$. Bottom right: Graph $G/\{v_5v_6,v_1v_2,v_3v_6\}$ is color-connected.}
    \label{fig:prism}
\end{figure}
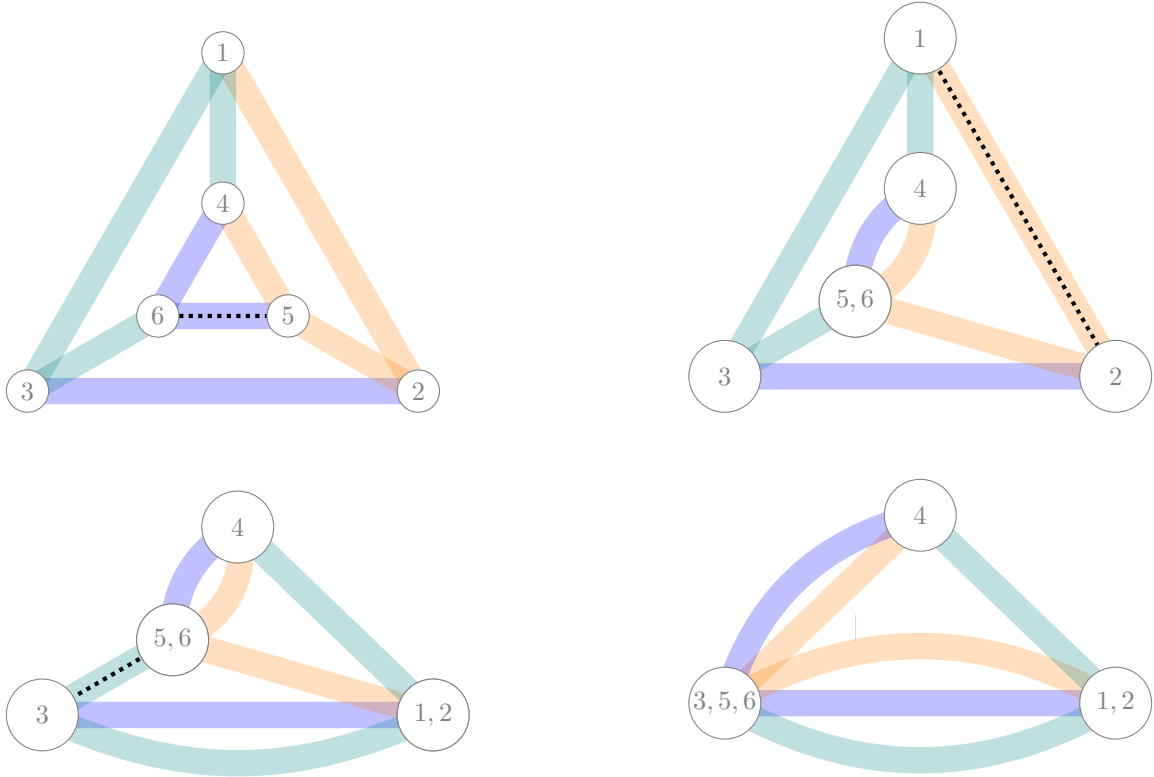

\begin{example}[Greedy Algorithm]\label{ex:greedy}

Consider an edge-colored graph $G$ with vertex set $V=\{v_1,v_2,v_3,v_4,v_5,v_6\}$ and $|\mathcal{C}|=3$.  Let color $c_1$ be green with edge set $E_1=\{v_1v_3,v_1v_4, v_3v_6\}$, color $c_2$ be purple with edge set $E_2=\{v_2v_3, v_4v_6, v_5v_6\}$, and color $c_3$ be orange with edge set $E_3=\{v_1v_2, v_2v_5, v_4v_5\}$, as shown in \cref{fig:prism}.  

To implement the Greedy Algorithm, we begin by computing the dip number of each edge in the graph, and we record those dip numbers in \cref{table:dip numbers for prism}.

    \begin{table}[h]\centering
    {\renewcommand{\arraystretch}{1.25}\begin{tabular}{|l|c|c|c|c|c|c|c|c|c|}
    \hline
    Edge &$v_1v_2$&$v_1v_3$&$v_1v_4$&$v_2v_3$&$v_2v_5$&$v_3v_6$&$v_4v_5$&$v_4v_6$&$v_5v_6$\\
    \hline
Dip Number&2&2&1&2&2&2&1&1&2\\
\hline
\end{tabular}}
\caption{Table of dip numbers for all edges of the graph in \cref{fig:prism}}\label{table:dip numbers for prism}
	\end{table}
Following the algorithm, we compute the potential of each of the edges with maximal dip number, i.e.,  the six edges with dip number $2$.  This is done in \cref{table:potential numbers for prism}.
 The six edges of dip number 2 in the graph in \cref{fig:prism} are listed in the first column of \cref{table:potential numbers for prism}.  Reading along a row of \cref{table:potential numbers for prism}, say row $uv$, each entry records the dip number of that column's edge in the graph $G/uv$.  For example, consider the last row of \cref{table:potential numbers for prism}, the row for edge $v_5v_6$.  Following along this row to the column with header $v_2v_5$, we find the entry $2$.  This tells us that if we change the edge $v_5v_6$ to wild, giving us the graph $G/v_5v_6$, then the dip number of the edge $v_2v_5$ in the graph $G/v_5v_6$ is 2.

    \begin{table}[h]\centering
    {\renewcommand{\arraystretch}{1.25}\begin{tabular}{ccccccccccc}
    Wild Edge&$v_1v_2$&$v_1v_3$&$v_1v_4$&$v_2v_3$&$v_2v_5$&$v_3v_6$&$v_4v_5$&$v_4v_6$&$v_5v_6$ & Potential \\
    \hline
$v_1v_2$ & 0 & 1 & 1 &1 &2 &2 &1 &1 &2 & $2^{(3)},1^{(5)},0^{(1)}$\\
$v_1v_3$ & 1 & 0 & 1 & 1 & 2& 2 &1 & 1 & 2 & $2^{(3)},1^{(5)},0^{(1)}$\\
$v_2v_3$ & 1 & 1 & 1 & 0 & 2 & 2 &1 &1 &2 & $2^{(3)},1^{(5)},0^{(1)}$\\
$v_2v_5$ & 2 & 2 & 1 & 2 &0 &1 &1 &1 &2 & $2^{(4)},1^{(4)},0^{(1)}$\\
$v_3v_6$ & 2 & 2 & 1 & 2 &1 &0 & 1 & 1 &2 & $2^{(4)},1^{(4)},0^{(1)}$\\
$v_5v_6$ & 2 & 2 & 1 &2 &2 & 2 & 0 &0 &0 & $2^{(5)},1^{(1)},0^{(3)}$\\
\hline
         
\end{tabular}}
\caption{Table computing potentials for all edges with dip number 2 in \cref{fig:prism}}\label{table:potential numbers for prism}
\end{table}

For each edge of dip number 2 in the graph in \cref{fig:prism}, we use the information in that edge's row of \cref{table:potential numbers for prism} to determine the potential of the edge, and this potential is recorded in the last column of \cref{table:potential numbers for prism}.   The edge with greatest potential is $v_5v_6$, so we color this edge wild. 

We now repeat step one with the graph $G/v_5v_6$.   Considering the $5$ edges in $G/v_5v_6$ with dip number $2$, the greatest potential of an edge is $2^{(1)},1^{(4)},0^{(3)}$. 
In fact, the four edges $v_1v_2$, $v_1v_3$, $v_2v_5$, and $v_3v_6$ all share this potential.  According to the Greedy Algorithm, we can choose which of these edges to color wild. We choose $v_1v_2$.

Now we have the graph $G/W$, where $W={\{v_1v_2,v_5v_6\}}$.  Among the edges in this graph, the highest dip number is $2$, and only the edge $v_3v_6$ has dip number $2$.  Thus we color the edge $v_3v_6$ wild. Let $W'=\{v_1v_2, v_3v_6, v_5v_6\}$.  Now $G/W'$ has a spanning tree in each color of $\mathcal{C}$, so $W'$ is a color-connecting wild set of $G$.  

In this case, $W'$ is not only a color-connecting wild set for $G$ but is actually an ideal wild set for $G$, as the ceiling lower bound of $G$ is $\left\lceil\frac{\cub(G)}{\ell-1}\right\rceil=\lceil\frac{6}{2}\rceil=3$, so we see that $\wild(G)=3$.

\end{example}

We will see in the next example, however, that the Greedy Algorithm does not always produce an ideal wild set for a graph $G$.

\begin{example}
    As an example in which the Greedy Algorithm produces a color-connecting wild set that is not an ideal wild set, consider the wheel graph in \cref{fig:counterex}, which has vertex set $V=\{v_i:0\leq i\leq 7\}$ and colored edge sets
    \begin{align*}
        E_1 &=\{v_0v_1,v_0v_3,v_0v_4\}, \\
        E_2 &=\{v_0v_2,v_0v_6,v_4v_5,v_6v_7\},\\ 
        E_3 &= \{v_0v_5,v_0v_7,v_1v_2,v_1v_7,v_5v_6\},\text{ and}\\ 
        E_4 &= \{v_2v_3,v_3v_4\}. 
    \end{align*}
    This graph has wild number $5$ with ideal wild set $W=\{v_0v_3,v_0v_7,v_1v_2,v_4v_5,v_5v_6\}$, but \cref{greedy} yields a color-connecting wild set of size $6$, with edges chosen in the following order: $v_1v_7$, $v_2v_3$, $v_4v_5$, $v_5v_6$, $v_6v_7$, $v_0v_4$.
    
\end{example}

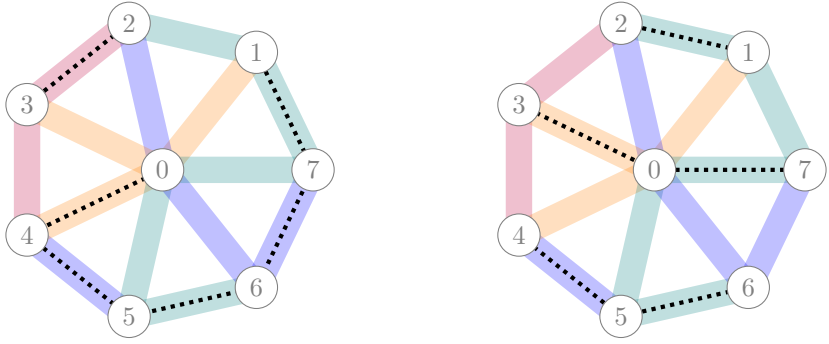
\begin{figure}
    \centering{}\hfill
    \begin{tikzpicture}
        \def \radius {8pt};
        \def \va {(1*360/7:2)};
        \def \vb {(2*360/7:2)};
        \def \vc {(3*360/7:2)};
        \def \vd {(4*360/7:2)};
        \def \ve {(5*360/7:2)};
        \def \vf {(6*360/7:2)};
        \def \vg {(7*360/7:2)};
        \def \vo {(0,0)};
        \draw[teal,highlight] \va--\vb;
        \draw[purple,highlight] \vb--\vc;
        \draw[purple,highlight] \vc--\vd;
        \draw[blue,highlight] \vd--\ve;
        \draw[teal,highlight] \ve--\vf;
        \draw[blue,highlight] \vf--\vg;
        \draw[teal,highlight] \vg--\va;
        \foreach \pt/\edgecolor in {\va/{orange},\vb/{blue},\vc/{orange},\vd/{orange},\ve/{teal},\vf/{blue},\vg/{teal}}
        \draw[\edgecolor,highlight] \pt--\vo;
        \draw[wild] \vb--\vc;
        \draw[wild] \vo--\vd--\ve--\vf--\vg--\va;
        \foreach \pt/\name in {\va/{1},\vb/{2},\vc/{3},\vd/{4},\ve/{5},\vf/{6},\vg/{7},\vo/{0}}
         \draw[gray,fill=white] \pt circle[radius=\radius] node {\footnotesize ${\name}$};
    \end{tikzpicture}\hfill
    \begin{tikzpicture}
        \def \radius {8pt};
        \def \va {(1*360/7:2)};
        \def \vb {(2*360/7:2)};
        \def \vc {(3*360/7:2)};
        \def \vd {(4*360/7:2)};
        \def \ve {(5*360/7:2)};
        \def \vf {(6*360/7:2)};
        \def \vg {(7*360/7:2)};
        \def \vo {(0,0)};
        \draw[teal,highlight] \va--\vb;
        \draw[purple,highlight] \vb--\vc;
        \draw[purple,highlight] \vc--\vd;
        \draw[blue,highlight] \vd--\ve;
        \draw[teal,highlight] \ve--\vf;
        \draw[blue,highlight] \vf--\vg;
        \draw[teal,highlight] \vg--\va;
        \foreach \pt/\edgecolor in {\va/{orange},\vb/{blue},\vc/{orange},\vd/{orange},\ve/{teal},\vf/{blue},\vg/{teal}}
        \draw[\edgecolor,highlight] \pt--\vo;
        \draw[wild] \va--\vb;
        \draw[wild] \vc--\vo--\vg;
        \draw[wild] \vd--\ve--\vf;
        \foreach \pt/\name in {\va/{1},\vb/{2},\vc/{3},\vd/{4},\ve/{5},\vf/{6},\vg/{7},\vo/{0}}
         \draw[gray,fill=white] \pt circle[radius=\radius] node {\footnotesize ${\name}$};
    \end{tikzpicture}\hfill{}
    \caption{The greedy algorithm produces a wild set of size six (left), but the wild number is five (right).}
    \label{fig:counterex}
\end{figure}

\section{Colors}\label{sec:colors}

In this section, we explore extremal cases of $\ell$, the number of colors used in the edge-coloring of $G$. Additionally, we identify sufficient conditions on the edge-coloring that guarantee the wild number of $G$ will exceed the component lower bound of $G$.
  In \cref{thm:extremal cases of colors} Part \ref{part:ell is 1}, we consider the case where $\ell=1$.  Elsewhere in this section, we assume $\ell\geq 2$. Throughout, we use the notation $\kappa_{\min}=\min_{1\leq i\leq \ell}\kappa(G_i)$ and $\kappa_{\max}=\max_{1\leq i\leq \ell}\kappa(G_i)$, so $\operatorname{clb}(G)=\kappa_{\max}-1$.

\begin{theorem}\label{thm:clb general connections}
Let $(G, \gamma)$ be an edge-colored graph with $n$ vertices, $m$ edges, and $\ell$ colors. 

If $s$ is the number of colors with $\kappa_{\max}$ components and 
$\kappa_{\max}-1>(\ell-s)(\kappa_{\max}-\kappa_{\min})$, then $\wild(G)>\kappa_{\max}-1$. Equivalently, if $\operatorname{clb}(G)>(\ell-s)(\kappa_{\max}-\kappa_{\min})$, then $\wild(G)>\operatorname{clb}(G)$.
\end{theorem}

\begin{proof}Suppose $(G, \gamma)$ is an edge-colored graph with $n$ vertices, $m$ edges, and $\ell$ colors.
The $s$ colors with $\kappa_{\max}$ components each require $\kappa_{\max}-1=\operatorname{clb}(G)$ connections, while the remaining $\ell-s$ colors that have at least $\kappa_{\min}$ components each need at least $\kappa_{\min}-1$ connections. Thus, any color-connecting wild set of $G$ must have a dip number of at least  $T=s(\kappa_{\max}-1)+(\ell-s)(\kappa_{\min}-1)$. The maximum dip number of an edge in $G$ is $\ell-1$ colors, so if $|W| \leq \kappa_{\max}-1$, then $\delta(W) \leq(\ell-1)(\kappa_{\max}-1)$, leaving at least $T-(\ell-1)(\kappa_{\max}-1)$ remaining connections to be made.
But 
\begin{align*}
   T - (\ell-1)(\kappa_{\max}-1) &= (s-\ell+1)(\kappa_{\max}-1) + (\ell-s)(\kappa_{\min}-1) \\
   &= \kappa_{max}-1 - (\ell-s)(\kappa_{\max}-1)+(\ell-s)(\kappa_{\min}-1)\\
   &=\kappa_{max}-1 - (\ell-s)(\kappa_{\max}-\kappa_{\min}).
\end{align*}
With the assumption that $\kappa_{\max}-1>(\ell-s)(\kappa_{\max}-\kappa_{\min})$, we see that more connections remain. Hence $\wild(G)>\operatorname{clb}(G)$.
\end{proof}

Our first corollary to \cref{thm:clb general connections} addresses the special case where the number of colors with $\kappa_{max}$ components is at least $\ell-1$ and the remaining color has at least two components. The corollary shows that in this case the wild number of $G$ must exceed the component lower bound of $G$.

\begin{corollary}\label{prop:most components geq k+1}
    Let $G$ be an edge-colored graph with $\ell$ colors.
    If $\kappa(G_i)=\kappa_{\max}$ for at least $\ell-1$ colors and $\kappa_{\min}>1$, then $\wild(G)>\operatorname{clb}(G)$.

\end{corollary}

\begin{proof}

Let $s$ be the number of colors with $\kappa_{\max}$ components. Assuming $s\geq \ell-1$, we have $0\leq \ell-s\leq 1$. Also, $\kappa_{\max}-\kappa_{\min}<\kappa_{\max}-1$ since $\kappa_{\min}>1$. Hence
$(\ell-s)(\kappa_{\max}-\kappa_{\min})<\kappa_{\max}-1$. By \cref{thm:clb general connections}, we have $\wild(G)>\operatorname{clb}(G)$.
\end{proof}

We say that a color $c_{i} \in \mathcal{C}$ is {\bf used exactly $t$ times} if $\vert E_{i}(G) \vert = t$, that is, $G$ has exactly $t$ edges with color $c_{i}$. 
Applying \cref{thm:clb general connections} to the special case where $\ell-1$ of the colors are used exactly once and where $G$ does not already have a spanning tree in the one remaining color, we identify a class of graphs for which the wild number is the extremal value $n-1$.

\begin{corollary}[]\label{thm: when n-2>ell-s}
    Suppose $(G, \gamma)$ is a graph with $n$ vertices, $m$ edges, and $s>0$ colors used exactly once in the graph, while each of the other $\ell-s$ colors used in the graph is used exactly twice.  If $n-2>\ell-s$, then $\wild(G)=n-1$.  
\end{corollary}

\begin{proof}
    By applying \cref{thm:clb general connections} with $\kappa_{\max}=n-1$, $\kappa_{\min}=n-2$, and the hypothesis $n-2>\ell-s$, we have $\wild(G)>n-2$. Since $\wild(G)\leq n-1$ always, we get $\wild(G)=n-1$.
\end{proof}

In the next theorem, we summarize some simple results for extremal cases of the number of colors.

\begin{theorem}\label{thm:extremal cases of colors}
    Let $(G,\gamma)$ be an edge-colored simple graph with $n>2$ vertices, $m$ edges, and $\ell$ colors.
    \begin{enumerate}
        \item\label{part:ell is 1} The smallest possible number of colors is $\ell=1$.  In this case, $\wild(G)=0$.
        \item\label{part:ell is 2} If $\ell=2$, then $\wild(G)=\operatorname{cub}(G)=\kappa(G_1)+\kappa(G_2)-2$.
        \item\label{part:ell is m-1} If $\ell=m-1>1$, then
        \[
        \wild(G)=\begin{cases}
            1 & \text{if $n=3$} \\
            n-1 & \text{if $n>3$.}
        \end{cases}
        \]
    \item\label{part:ell is m} The largest possible number of colors is $\ell=m$.  In this case, $\wild(G)=n-1$.
    \end{enumerate}
\end{theorem}

\begin{proof}
    First suppose $\ell=1$.  Then every edge in $G$ is already the same color, so $G$ has a spanning tree in this one color, and $\wild(G)=0$.  Here the number of necessary changes in edge colors is as small as possible, meaning the wild number of a graph is minimized when the number of colors is minimized.  
    
    Now suppose $\ell=2$.  Changing any edge to wild helps exactly one color, so we must change a total of $\kappa(G_1)+\kappa(G_2)-2$ edges.     

    Next suppose $\ell=m-1$ and $m>2$. If $n=3$, then because $G$ is a simple connected graph, $G = C_{3}$. So, in this case, from \cref{prop:cycle wild number}, $\wild(G)=1$.  Now suppose $n>3$. By applying \cref{thm:clb general connections} with $\kappa_{\max}=n-1$, $\kappa_{\min}=n-2$, and $\ell-s=1$, we have $\wild(G)>n-2$. Since $\wild(G)\leq n-1$ always, we get $\wild(G)=n-1$.
 
    Lastly, if $\ell=m$, then every edge in $G$ is a  different color, so by \cref{prop:most components geq k+1} with $\kappa_{\max}=\kappa_{\min}=n-1>1$, we get $\wild(G)=n-1$. In this case, the number of necessary changes of edge colors is as large as possible, meaning the wild number of a graph is maximized when the number of distinct edge colors is maximized.
\end{proof}

In what follows, we let $m_{\clb}$ denote the number of edges in $G$ whose color meets the component lower bound $\operatorname{clb}(G)$, that is, $m_{\clb} = \sum_{i \in I} \vert E_{i}(G) \vert $, where $I = \{ i : \kappa(G_i)  = \kappa_{\max}\}$. 

\begin{theorem}[]\label{thm:edges meeting clb}
Let $(G, \gamma)$ be an edge-colored graph with $n$ vertices and $m$ edges.  If $\clb(G)>m-m_{\clb}$, then $\wild(G)>\operatorname{clb}(G)$.  
\end{theorem}
\begin{proof}
    Suppose $\wild(G)\leq \operatorname{clb}(G)$, and therefore $\wild(G)=\operatorname{clb}(G)$ as the wild number cannot be less than a lower bound.  Let $k=\clb(G)$.  Suppose there exists an ideal wild set $W$ of size $k$. Then every edge in $W$ must help each of the colors that need $k$ connections. So, since an edge cannot help its own color, $W$ cannot contain edges whose color needs $k$ connections. Thus we have $m-m_{\clb}$ edges to choose from when building $W$, so it must be that $k=\operatorname{clb}(G)\leq m-m_{\clb}$.
\end{proof}

As a corollary, we have another sufficient condition for a graph to have wild number $n-1$.

\begin{corollary}\label{cor: when m-s<n-2}
    Suppose $(G, \gamma)$ is an edge-colored graph with $n$ vertices, $m$ edges, and $s$ colors that are each used exactly once in the graph $G$. If $n-2>m-s$, then $\wild(G)=n-1$.
\end{corollary}
\begin{proof}
    Though this result can be proven directly, we will show that it is a special case of \cref{thm:edges meeting clb}.
    Suppose $(G,\gamma)$ is an edge-colored graph where $s$ colors are each used exactly once. Then $\clb(G) = \max_{1\leq i\leq \ell}(\kappa(G_i)-1)=n-2$ and $s=m_{\clb}$. Assuming $n-2>m-s$, we have $\clb(G)>m-m_{\clb}$.
    So we can apply \cref{thm:edges meeting clb} to conclude that $\wild(G)>\operatorname{clb}(G)=n-2$.  Hence $\wild(G)=n-1$.
\end{proof}

\section{Conclusion}\label{sec:conclusion}
We end with a list of open questions which we invite our dear reader to explore.  We are confident that any graph theory enthusiast, from undergraduate students to those more experienced, will enjoy working on any of the following questions. Moreover, as this is a novel concept in graph theory, we are also certain that readers can develop their own interesting research directions.

In  \cref{sec:families} and \cref{sec:colors}, there are examples of families of graphs where $\wild(G)=n-1$, or $\wild(G)=1$.  These are extreme cases that can be considered, similar to the results on wild number zero discussed in \cref{sec:classic}. 

\begin{question}
    Can we classify the graphs $G$ such that $\wild(G)=1$ or $\wild(G)=n-1$?
\end{question}

In \cref{sec:families}, one may note that the wild number of a tree or cycle depends only on the number of vertices and the number of colors used.  This leads to the following question:
\begin{question}
Are there other graphs, or families of graphs, for which the wild number can be determined based simply on the number of colors used and the number of vertices, regardless of the particular edge-coloring of the graph?
\end{question}
\begin{question}
     What if we only considered edge-colored graphs for which the edge coloring is proper?  What could we say then about the wild number of the edge-colored graph? Could the wild number be determined based on certain characteristics of the proper edge-colored graph?
\end{question}

In \cref{sec:bounds}, \cref{greedy} is introduced, along with an example where the algorithm does not provide the wild number of the graph.  In our experience, with almost all of the examples we considered, the algorithm did lead us to the correct wild number of the graph. 
\begin{question}
    When does \cref{greedy} produce an ideal wild set?  Does it work for certain families of graphs? Does it work for certain edge-colorings? 
\end{question}

\begin{question}
    How far off can \cref{greedy} be from the actual wild number? Is it ever off by more than one? or two?\ldots
\end{question}

Given an edge-colored graph $G$, there are times when a certain edge \emph{must} be colored wild to produce an ideal wild set.  Is there a way to find all of these edges? Possibly, answering the following questions could help to determine these edges. 
\begin{question}  
    Can we characterize the edges $uv$ for which $\wild(G) =\wild(G/uv)=\wild(G)-1$?
\end{question}

\begin{question}
   Given an edge-colored graph, which edges are contained in every (or no) ideal wild set?
\end{question}

The last questions posed involve making connections to more classical graph theory questions. 

\begin{question}
     What is the interaction between color-connecting wild sets  and vertex coverings?  When is a color-connecting wild set a vertex covering?  When is a vertex covering a color-connecting wild set? When is it ideal?
\end{question}

\begin{question}
If $H$ is a subgraph of an edge-colored graph $G$, then how are $\wild(H)$ and $\wild(G)$ related? 
\end{question}

\subsection*{Acknowledgments} Part of this research was performed while the authors were visiting the Mathematical Sciences Research Institute (MSRI), now becoming the Simons Laufer Mathematical Sciences Institute (SLMath), which is supported by the National Science Foundation (Grant No.~DMS-1928930). In particular, the authors participated in SLMath's Summer Research in Mathematics (SRiM) program in 2023 and through generous post-programmatic support, the authors continued to work on the project through completion over the last two summers.  We thank SLMath and the NSF wholeheartedly for their generous support of our joyful and fruitful collaboration. 
\bibliographystyle{alphaurl}
\bibliography{WildBib}

\end{document}